\documentclass[12pt]{amsart}
\usepackage[utf8]{inputenc}

\usepackage{amssymb}
\usepackage{bm}
\usepackage{graphicx}
\usepackage[centertags]{amsmath}
\usepackage{amsfonts}
\usepackage{amsthm}
\usepackage{enumerate}
\usepackage{tocvsec2}
\usepackage{xcolor}
\usepackage[margin=1in]{geometry}

\newtheorem{theorem}{Theorem}[section]
\newtheorem*{theorem*}{Theorem}

\newtheorem{corollary}[theorem]{Corollary}
\newtheorem{lemma}[theorem]{Lemma}
\newtheorem{rem}[theorem]{Remark}

\newtheorem{proposition}[theorem]{Proposition}

\theoremstyle{definition}

\newenvironment{remark}[1][Remark]{\begin{trivlist}
\item[\hskip \labelsep {\bfseries #1}]}{\end{trivlist}}

\newcommand{\ord}{\textbf{Ord}}

\newcommand{\rr}{\mathbb{R}}
\newcommand{\nn}{\mathbb{N}}

\newcommand{\ee}{\varepsilon}

\newcommand{\aaa}{\mathcal{A}}
\newcommand{\bbb}{\mathcal{B}}

\newcommand{\ccc}{\mathcal{C}}

\newcommand{\fff}{\mathcal{F}}
\newcommand{\mmm}{\mathcal{M}}

\newcommand{\hhh}{\mathcal{H}}

\newcommand{\mt}{\mathcal{MT}}
\newcommand{\ttt}{\mathcal{T}}
\newcommand{\nnn}{\mathcal{N}}

\newcommand{\cat}{^\smallfrown}
\newcommand{\Ban}{\textbf{Ban}}

\begin{document}

\title[Alternate description of the Szlenk index]{An alternate description of the Szlenk index \\ with applications}

\author{RM Causey}

\begin{abstract} We discuss an alternate method for computing the Szlenk index of an arbitrary $w^*$ compact subsets of the dual of a Banach space.  We discuss consequences of this method as well as offer simple, alternative proofs of a number of results already found in the literature.  

\end{abstract}

\maketitle

\section{Introduction}

Since its inception, Banach space theory has employed ordinal indices.  One of the most well-known indices is that introduced by Szlenk \cite{Sz}.  The index was originally used to prove the non-existence of a Banach space having separable dual which is universal for the class of Banach spaces having separable dual. Since its introduction, the standard definition of the Szlenk index has become different than that originally given by Szlenk, although the two definitions yield the same index for any separable Banach space not containing $\ell_1$ isomorphically. Because we are interested in computing the indices of operators on domains which may contain isomorphs of $\ell_1$, or the Szlenk index of non-separable Banach spaces, we use the now-common definition of the Szlenk index, and not the original definition.   Since Szlenk introduced his index, it has seen a number of uses \cite{L2} and has been the subject of significant study.  The Szlenk index can be defined for any $w^*$ compact subset of the dual of a Banach space.  The Szlenk index of a Banach space is then defined to be the Szlenk index of the closed unit ball of the dual space.  In \cite{AJO}, the authors established an alternative method for computing the Szlenk index of a Banach space whenever that Banach space is separable and does not contain an isomorphic copy of $\ell_1$.  In \cite{C2}, the author provided a partial extension of the methods of \cite{AJO} to provide an alternative characterization of the Szlenk index of certain $w^*$ compact subsets of the dual of a separable Banach space.  In this work, we provide a complete extension of these results to establish an alternative method, analogous to those used in \cite{AJO} and \cite{C2}, to compute the Szlenk index of any $w^*$ compact subset of the dual of a Banach space.  The methods in these works used certain minimal structures, namely the fine Schreier families, to witness the size of indices.  The use of the fine Schreier families, however, limits the applicability of these methods to those spaces in which pertinent properties (for example, $w^*$ convergence, or pointwise convergence on a subset of the dual space) are sequentially determined.  This work advances previous results in three ways: Given a Banach space $X$, with a description analogous to that appearing in \cite{AJO} used to compute the Szlenk index of $B_{X^*}$,  we have been able to compute the Szlenk index of any $w^*$ compact subset of $X^*$, while being able to do so without the assumptions of separablity of $X$ or that $\ell_1$ does not embed into $X$.  

 In this work, we introduce a convenient method of constructing minimal structures (analogues of the Schreier and fine Schreier families) which are able to take into account, for example, non-metrizability of the $w^*$ topology on the unit ball of the dual of a Banach space.  These minimal structures involve combining directed sets with minimal trees introduced by the author in \cite{C2}, and we believe this method of constructing minimal structures could be of independent interest.  These structures facilitate short, simple proofs of some new results, as well as new proofs of results already existing in the literature. After we provide an alternative characterization of the Szlenk index and prove that it is equivalent to the more common definition involving slicings, we are able to offer all of our proofs of both new and old results without ever referring again to the slicing definition.

\section{Definitions and the main theorem}

We follow standard Banach space notation. We will assume $X$ is a real Banach space, although the results apply as well to complex Banach spaces with appropriate modifications which we indicate along the way. If $X$ is a Banach space, we let $S_X$, $B_X$ denote the unit sphere and closed unit ball of $X$, respectively. If $S$ is a subset of $X$, we let $[S]$ denote the closed span of $S$. By a subspace of $X$, we mean a closed subspace of $X$.  By an operator between Banach spaces, we mean a bounded linear operator.  We let $\nn=\{1, 2, \ldots\}$ and $\nn_0=\{0\}\cup \nn$.  We let $\ord$ denote the class of ordinal numbers. We let $\Ban$ denote the class of all Banach spaces.  If $\Lambda$ is a set, we let $\Lambda^{<\nn}$ denote the finite sequences in $\Lambda$.  We include in $\Lambda^{<\nn}$ the sequence of length $0$, denoted $\varnothing$. We let $2^\Lambda$ denote the power set of $\Lambda$, $[\Lambda]^{<\nn}$ the finite subsets of $\Lambda$. If $s,t\in \Lambda^{<\nn}$, we let $s\cat t$ denote the concatenation of $s$ with $t$ listing the members of $s$ first. For $t\in \Lambda^{<\nn}$, we let $|t|$ denote the length of $t$.  We freely identify $\Lambda$ with sequences of length $1$ in $\Lambda^{<\nn}$.  That is, if $t$ is a sequence of length $1$, say $t=(x)$, we will write $x\cat s$ in place of $(x)\cat s$, etc.   We order $\Lambda^{<\nn}$ by letting $s\preceq t$ if $s$ is an initial segment of $t$.  For $T\subset \Lambda^{<\nn}$, we let $MAX(T)$ denote the set of maximal elements of $T$ with respect to the order $\preceq$.  Given $T\subset \Lambda^{<\nn}$, we say $T$ is a \emph{tree} if $T$ is downward closed with respect to the order $\preceq$.  We say $T\subset \Lambda^{<\nn}\setminus \{\varnothing\}$ is a $B$-\emph{tree} provided that $T\cup \{\varnothing\}$ is a tree.  All definitions below regarding trees can be relativized to $B$-trees.   We say a tree $T$ is \emph{hereditary} if for any $t\in T$ and any subsequence $s$ of $t$, $s\in T$.   We say a map $f:T\to T_0$ between trees is \emph{monotone} provided that for each $s,t\in T$ with $s\prec t$, $f(s)\prec f(t)$.   For any $t\in \Lambda^{<\nn}$ and any integer $n$ with $0\leqslant n\leqslant |t|$, we let $t|_n$ denote the initial segment of $t$ having length $n$. We let $p(t)=t|_{|t|-1}$ for each $t\in \Lambda^{<\nn}\setminus \{\varnothing\}$.  That is, for $t\in \Lambda^{<\nn}\setminus \{\varnothing\}$, $p(t)$ denotes the largest proper initial segment of $t$.  If $\Lambda_1$, $\Lambda_2$ are sets, we identify $(\Lambda_1\times \Lambda_2)^{<\nn}$ with $\{(s,t)\in \Lambda_1^{<\nn}\times \Lambda_2^{<\nn}: |s|=|t|\}$.  In this case, we identify $(\varnothing, \varnothing)$ with $\varnothing$.  

We next recall the \emph{order} of a tree. If $T$ is a tree on $\Lambda$, then we let $T'$ consist of all members of $T$ which are maximal in $T$ with respect to $\preceq$. We call $T'$ the \emph{derived tree} of $T$. We note that $T'$ is a tree (resp. hereditary tree) if $T$ is a tree (resp. hereditary tree). We then define the higher order derived trees $T^\xi$, $\xi\in \ord$, as follows: $$T^0=T,$$ $$T^{\xi+1}=(T^\xi)',$$ and if $T^\zeta$ has been defined for each $\zeta<\xi$, $\xi$ a limit ordinal, we define $$T^\xi=\bigcap_{\zeta<\xi} T^\zeta.$$  If there exists an ordinal $\xi$ so that $T^\xi=\varnothing$, we let $o(T)$ be the minimum such ordinal, and call $o(T)$ the \emph{order} of $T$.  If there is no such ordinal, we write $o(T)=\infty$.  To save a great deal of writing, we will agree that for $\xi\in \ord\cup \{\infty\}$, $\xi\infty = \infty\xi=\xi+\infty=\infty+\xi=\infty$.  Moreover, for any $\xi\in \ord$, we agree that $\xi<\infty$.  

If $T$ is a tree on $\Lambda$ and $t\in \Lambda^{<\nn}$, we let $T(t)=\{s\in \Lambda^{<\nn}: t\cat s\in T\}$.  This is a tree, empty if and only if $t\notin T$, hereditary if $T$ is hereditary.  It is easy to see that for any tree $T$ on $\Lambda$, $t\in \Lambda^{<\nn}$, and $\xi\in \ord$, $T^\xi(t)=(T(t))^\xi$.  It is also a standard induction argument that for any $\xi, \zeta\in \ord$ and any tree $T$ on $\Lambda$, $(T^\xi)^\zeta=T^{\xi+\zeta}$.

We next define a notion related to order and derived trees. Whereas a sequence $t\in T$ need only have one proper extension in $T$ to be admitted into $T'$, one is frequently interested in those members $t$ of $T$ for which there exists a collection $(x_U)_{U\in D}$ satisfying some property (such as being a weakly null net, as will be our primary interest) so that all proper extensions $t\cat x_U$ of $t$ lie in $T$. Given a subset $\hhh\subset \Lambda^{<\nn}$ and $\varnothing \neq D\subset 2^\Lambda$, we let $$(\hhh)'_D=\{t\in \hhh: \forall U\in D, \exists x\in U (t\cat x\in \hhh)\}.$$ We note that if $\hhh$ is a hereditary tree, $(\hhh)_D'$ is a hereditary tree as well. However, if $\hhh$ fails to be hereditary, $(\hhh)_D'$ may fail to be a tree.   Next, we define $(\hhh)_D^\xi$ for $\xi\in \ord$ by transfinite induction.  We let $$(\hhh)_D^0=\hhh,$$ $$(\hhh)_D^{\xi+1}= ((\hhh)_D^\xi)'_D,$$ and if $\xi$ is a limit ordinal and $(\hhh)_D^\zeta$ has been defined for each $\zeta<\xi$, we let $$(\hhh)_D^\xi=\bigcap_{\zeta<\xi}(\hhh)_D^\zeta.$$  If there exists $\xi\in \ord$ so that $(\hhh)_D^\xi=\varnothing$, we let $o_D(\hhh)$ be the minimum such ordinal.   If no such $\xi$ exists, we write $o_D(\hhh)=\infty$. If $D=\{\Lambda\}$, this recovers the usual notions of derived trees and the order of a tree. As with the usual notion of derived trees, $((\hhh)^\xi_D)^\zeta_D=(\hhh)^{\xi+\zeta}_D$ and for any $t\in \Lambda^{<\nn}$, $(\hhh(t))^\xi_D=(\hhh)^\xi_D(t)$.  If $X$ is a Banach space and $D$ is a weak neighborhood basis at zero, we write $(\hhh)_w'$, $(\hhh)_w^\xi$, and $o_w(\hhh)$ in place of $(\hhh)_D'$, $(\hhh)_D^\xi$, and $o_D(\hhh)$. We refer to the derivation $\hhh\mapsto (\hhh)_w'$ as the \emph{weak derivative}, and $o_w(\cdot)$ as the \emph{weak order}.    It is easy to see that if $D$, $D_0$ are two weak neighborhood bases at zero, the $D$ and $D_0$ derivations, and therefore the $D$ and $D_0$ orders, coincide, and there is no ambiguity in defining the weak derivative and weak order through a fixed weak neighbhood basis at zero.  

We note that the definition above is related to the notion of an $S$-derivative defined in \cite{OSZ}, which uses sequences.  While the definition above is not a direct generalization of the notion of an $S$-derivative, we note that all examples listed there are examples of the derivation defined here as well.  However, since we hope to extend previous results to the case of a non-separable Banach space, it is impossible to offer our characterization using sequences.   

We next recall the slicing definition of the Szlenk index. This will be our definition of the Szlenk index, although it differs from that originally given by Szlenk.  The definitions coincide when $X$ is a separable Banach space not containing $\ell_1$.   If $X$ is a Banach space, $K\subset X^*$ is $w^*$-compact, and $\ee>0$, we let $$s_\ee(K)=K\setminus \bigcup \{W\subset X^*: W \text{\ is\ }w^*\text{\ open}, \text{diam}(K\cap W)\leqslant \ee\}.$$  Of course, $s_\ee(K)$ is also $w^*$ compact.  We define the higher order derived sets by $$s_\ee^0(K)=K,$$ $$s_\ee^{\xi+1}(K)=s_\ee(s_\ee^\xi(K)),$$ and if $s_\ee^\zeta(K)$ has been defined for each $\zeta<\xi$, we let $$s_\ee^\xi(K)=\bigcap_{\zeta<\xi}s_\ee^\zeta(K).$$  We let $Sz_\ee(K)$ denote the minimum ordinal $\xi$ so that $s_\ee^\xi(K)=\varnothing$ if such an ordinal $\xi$ exists, and $Sz_\ee(K)=\infty$ otherwise.  We let $Sz(K)=\sup_{\ee>0} Sz_\ee(K)$.  

For a Banach space $X$, $\delta>0$, and $K\subset X^*$, let $$(K, \delta)=\Bigl\{\{x\in X: \forall x^*\in F \text{\ }(|x^*(x)|< \delta)\}: F\subset K \text{\ is finite}\Bigr\}.$$ Let $$\mmm=\Bigl\{\{x\in X: \forall x^*\in F \text{\ }(|x^*(x)|<\delta)\}: \delta>0, F\subset X^* \text{\ is finite}\Bigr\},$$    $$\nnn=\Bigl\{\{x^*\in X^*: \forall x\in F \text{\ }(|x^*(x)|<\delta) \}: \delta>0, F\subset X \text{ is finite}\Bigr\}.$$ Of course, these sets depend upon the Banach space $X$, but $X$ will be clear in most contexts.  If there is danger of ambiguity, we will write $\mmm(X)$ in place of $\mmm$, etc.  

 Observe that if we order all power sets by reverse inclusion, the sets defined above are all directed sets and closed under finite intersections. Moreover, $\mmm$ is a weak neighborhood basis at $0$ in $X$.    We will treat these sets as directed sets throughout.  

Throughout this work, for $K\subset X^*$ non-empty and $w^*$ compact, $\ee>0$, we let $$\hhh^K_\ee=\{t\in B_X^{<\nn}: \exists x^*\in K (x^*(x)\geqslant \ee \text{\ }\forall x\in t )\}.$$ We include the empty sequence in $\hhh^K_\ee$. We remark that for any $K\subset X^*$, any $\ee>0$, and any ordinal $\xi$, any convex block of a member of $(\hhh^K_\ee)_w^\xi$ is also a member of $(\hhh^K_\ee)_w^\xi$.  As with the sets $\mmm$, $\nnn$, etc., $\hhh^K_\ee$ depends upon the Banach space $X$ to which we omit direct reference.  In all contexts, it will be clear from the set $K$ in which Banach space the members of $\hhh^K_\ee$ lie.    

We remark that by the geometric version of the Hahn-Banach theorem, if $K=B_{X^*}$, the sequence $(x_i)_{i=1}^n$ lies in $\hhh^K_\ee$ if and only if every convex combination of $(x_i)_{i=1}^n$ has norm at least $\ee$.  For this reason, the index associated to the case $K=B_{X^*}$ has been referred to in the literature as the $\ell_1^+$ index \cite{AJO}.  More generally, for $\varnothing \neq K\subset X^*$ $w^*$ compact, we may define $|x|_K=\max_{x^*\in K} x^*(x)$.  It is obvious that if $(x_i)_{i=1}^n$ is such that there exists $x^*\in K$ so that $x^*(x_i)\geqslant \ee$ for each $1\leqslant i\leqslant n$, then any convex combination $x$ of $(x_i)_{i=1}^n$ has $|x|_K\geqslant \ee$.  If $K$ is symmetric and convex, the converse is also true.  This is seen by applying the geometric version of Hahn-Banach to separate the $\tau$-open convex set $\{x: |x|_K<\ee\}$ from the convex hull of $(x_i)_{i=1}^n$ by a linear functional $f:X\to \rr$ which is $\tau$-continuous, where $\tau$ is the topology on $X$ given by the seminorm $|\cdot|_K$. It is straightforward to verify in this case that those functionals $f:X\to \rr$ which are $\tau$-continuous are precisely those functionals $f\in X^*$ so that $|f|_K^*:=\sup_{|x|_K\leqslant 1} |f(x)|$ is finite and that $|f|_K^*\leqslant 1$ if and only if $f\in K$.  Therefore if we can separate $A$ from $B$ with a $\tau$-continuous functional $f:X\to \rr$, we may assume $|f|_K^*=1$, so $f\in K$, and that $\sup_{x\in A} f(x)\leqslant \inf_{x\in B}f(x)$.  It is immediate from the definitions that in this case, $\sup_{x\in A} f(x)= \ee$, whence $\ee\leqslant f(x_i)$ for each $1\leqslant i\leqslant n$.   We isolate this observation for future use.

\begin{rem} If $K\subset X^*$ is $w^*$ compact, non-empty, symmetric, and convex, and if $t\in B_X^{<\nn}$, then $t\in \hhh^K_\ee$ if and only if there exists $x^*\in K$ so that $x^*(x)\geqslant \ee$ for each $x\in t$.  We consider the empty sequence to satisfy both of these conditions.  

\label{useful remark}

\end{rem}

We note that in the complex case, we may define $\hhh^K_\ee$ similarly, except taking real parts of $x^*(x_i)$.  In this case, a similar characterization of membership in $\hhh^K_\ee$ exists using the appropriate complex version of the Hahn-Banach theorem.  We leave it to the reader to make the adjustments of the results below in the complex case.

We are now ready to state the main result.  

\begin{theorem} Suppose $K\subset X^*$ is $w^*$ compact and non-empty.  For any $\xi\in \emph{\ord}$, the following are equivalent: \begin{enumerate}[(i)]\item There exists $\ee>0$ so that $Sz_\ee(K)>\xi$.  \item There exists $\ee>0$ so that $o_w(\hhh^K_\ee)>\xi$. \item There exist $0<\delta<\ee$ so that $o_{(K, \delta)}(\hhh^K_\ee)>\xi$. \end{enumerate}
In particular, for any $w^*$ compact, non-empty subset $K$ of $X^*$, $$Sz(K)=\sup_{\ee>0}o_w(\hhh^K_\ee)=\sup_{\ee>\delta>0} o_{(K, \delta)}(\hhh^K_\ee).$$  

\label{main theorem}

\end{theorem}

Note that for $\xi=0$, each of the three conditions above is always true, and so that case follows.  We will only consider the non-trivial case $\xi>0$.  

Of course, since $(K, \delta) \subset \mmm$, for any hereditary tree $\hhh$ on $X$ and any $\xi\in \ord$, $$(\hhh)_w^\xi \subset  (\hhh)^\xi_{(K, \delta)},$$ whence $$o_w(\hhh)\leqslant  o_{(K,\delta)}(\hhh).$$  Thus $(ii)\Rightarrow (iii)$ of Theorem \ref{main theorem} is trivial.  We require some preliminaries for the remaining implications.

\section{Minimal structures}   In \cite{C}, the following trees were introduced.  We let $$\mt_0=\{\varnothing\},$$ $$\mt_{\xi+1}=\{\varnothing\}\cup \{(\xi+1)\cat t: t\in \mt_\xi\},$$ and if $\mt_\zeta$ has been defined for each $\zeta<\xi$, $\xi$ a limit ordinal, we let $$\mt_\xi=\bigcup_{\zeta<\xi} \mt_{\zeta+1}.$$  We also let $\ttt_\xi=\mt_\xi\setminus \{\varnothing\}$. Note that if $\xi$ is a limit ordinal, $\ttt_\xi=\cup_{\zeta<\xi}\ttt_{\zeta+1}$ is a totally incomparable union, since every member of $\ttt_{\zeta+1}$ is an extension of $(\zeta+1)$.   The following modification will be the primary tool of this work.  Given $D\subset 2^\Lambda$ and $\xi\in \ord$, we let $$\mt_\xi^D= \{(t, \sigma)\in ([1,\xi]\times D)^{<\nn}: t\in \mt_\xi\}$$ and $$\ttt_\xi^D= \{(t, \sigma)\in ([1,\xi]\times D)^{<\nn}: t\in \ttt_\xi\}.$$ Note that $\mt^D_\xi$ is a tree on $[1, \xi]\times D$ and $\ttt_\xi^D$ is a $B$-tree on $[1, \xi]\times D$.  If $D=\{\Lambda\}$, $\ttt_\xi^D$ is naturally isomorphic as a $B$-tree to $\ttt_\xi$.  Just as the trees $\ttt_\xi$ have been used to witness the order $o(T)$ of a tree $T$ \cite{C, BCF}, the tree $\ttt_\xi^D$ can naturally and easily be used to measure the order $o_D(\hhh)$ of a hereditary tree $\hhh$.  We observe that for any $0\leqslant \zeta\leqslant \xi$, $$(\mt_\xi^D)^\zeta=\{(t, \sigma)\in \mt_\xi^D: t\in \mt_\xi^\zeta\}$$ and $$(\ttt^D_\xi)^\zeta=\{(t, \sigma)\in \ttt_\xi^D: t\in \ttt_\xi^\zeta\}.$$ These statements can be verified easily by induction.  In particular, $\mt_\xi^D$ and $\ttt_\xi^D$ are well-founded and $o(\ttt_\xi^D)= o(\ttt_\xi)=\xi$ \cite{C}.

Recall that if $X$ is an understood Banach space, $\mmm$ and $\nnn$ are fixed weak and $w^*$ neighborhood bases at $0\in X$ and $0\in X^*$, respectively.  Because we will be frequently using $\ttt_\xi^\mmm$ and $\mt_\xi^\nnn$, we let $\aaa_\xi=\ttt_\xi^\mmm$ and $\bbb_\xi=\mt^\nnn_\xi$.   

The following is a modification of the corresponding result from \cite{AJO} to the non-separable case.  

\begin{lemma} Fix $K\subset X^*$ $w^*$ compact, $\ee>0$, $x^*\in X^*$, and $\xi\in \emph{\ord}$.  Then if $x^*\in s_\ee^\xi(K)$, there exists $(f_\beta)_{\beta\in \bbb_\xi}\subset K$ so that \begin{enumerate}[(i)]\item $f_\varnothing = x^*$, \item for each $t\in \ttt_\xi$, $\sigma\in \nnn^{<\nn}$ with $|\sigma|+1=|t|$, and each $U\in \nnn$, $\|f_{(p(t), \sigma)} - f_{(t, \sigma\cat U)}\|>\ee/2,$ and $f_{(p(t), \sigma)} - f_{(t, \sigma\cat U)}\in U.$  \end{enumerate}

\label{lemma1}

\end{lemma}

Recall that for $t\in \ttt_\xi$, $p(t)$ denotes the largest proper initial segment of $t$.  Then the collection $\{(t, \sigma\cat U):U\in \nnn\}$ is the set of all minimal proper extensions of $(p(t), \sigma)$ in $\bbb_\xi$.  

\begin{proof} By induction on $\xi$.  The $\xi=0$ case is trivial, since $\bbb_0=\{\varnothing\}$.  

Suppose the result holds for a given $\xi$.  Then if $x^*\in s_\ee^{\xi+1}(K)$, for each $U\in \nnn$, $\text{diam}(s_\ee^\xi(K)\cap (x^*+U))>\ee$.  This means we can find $g_U, h_U\in s^\xi_\ee(K)\cap (x^*+U)$ so that $\|h_U-g_U\|>\ee$.  Then we can choose $x^*_U\in \{g_U, h_U\}$ so that $\|x^*_U-x^*\|>\ee/2$.  This means $x^*_U\in s_\ee^\xi(K)$ and $x^*_U-x^*\in U$.  By the inductive hypothesis, for each $U\in \nnn$, there exists $(f^U_\beta)_{\beta\in \bbb_\xi}$ satisfying properties $(i)$-$(iii)$ with $x^*$ replaced by $x^*_U$.  We define $(f_\beta)_{\beta\in \bbb_{\xi+1}}$ as follows: Let $f_\varnothing = x^*$.  For $t\in \ttt_{\xi+1}$, we can write $t=(\xi+1)\cat s$ for some $s\in \mt_\xi$.  Then for such $t$, and for $\sigma\in \nnn^{<\nn}$ with $|\sigma|=|s|$, we let $$f_{(\xi+1, U)\cat(s, \sigma)}= f^U_{(s, \sigma)}.$$  It is straightforward to check that the requirements are satisfied.  

Last, assume the result holds for every $\zeta<\xi$, $\xi$ a limit ordinal.  If $x^*\in s_\ee^\xi(K)$, then $x^*\in s_\ee^{\zeta+1}(K)$ for each $\zeta<\xi$.  This means that for each $\zeta<\xi$, we can choose $(f^\zeta_\beta)_{\beta\in \bbb_{\zeta+1}}$ to satisfy $(i)$-$(iii)$ with $f^\zeta_\varnothing = x^*$ for each $\zeta<\xi$.  Then let $f_\varnothing = x^*$ and for $t\in \ttt_\xi$, note that since $\ttt_\xi=\cup_{\zeta<\xi}\ttt_{\zeta+1}$ is a disjoint union, $t\in \ttt_{\zeta+1}$ for some unique $\zeta$.  Then let $f_{(t, \sigma)}= f^\zeta_{(t, \sigma)}$.  Again, $(f_\beta)_{\beta \in \bbb_\xi}$ clearly satisfies the requirements.

\end{proof}

We remark here that the following slight improvement suggests itself.  It is an easy modification of the above method, and it will not be used in the sequel, so we omit the proof.  It is, however, an example of the flexibility of our method for constructing minimal trees.  

\begin{lemma} Let $\ccc_\xi=\mt_\xi^{\nnn\times \{\pm 1\}}$. Fix $\ee>0$.  For $K\subset X^*$ $w^*$ compact and $x^*\in X^*$, $x^*\in s_\ee^\xi(K)$ if and only if there exists $(x^*_\beta)_{\beta \in \ccc_\xi}$ so that \begin{enumerate}[(i)]\item $x^*_\varnothing=x^*$, \item for $t\in \ttt_\xi$, and $\beta= (t, \sigma, \ee)\in \ccc_\xi$, $x^*_{p(\beta)}- x^*_\beta\in V$, \item for $t\in \ttt_\xi$, $\sigma\in \nnn^{<\nn}$, $\tau\in \{\pm 1\}^{<\nn}$ with $|t|=|\sigma|+1=|\tau|+1$ and $V\in \nnn$, $$\|x^*_{(\tau, \sigma\cat V, \tau\cat 1)}- x^*_{(\tau, \sigma \cat V, \tau\cat -1)}\|> \ee.$$

\end{enumerate}
\end{lemma}

The following should be compared to Proposition $5$ of \cite{OSZ}.  

\begin{lemma} Let $\Lambda$ be a non-empty set.  For any hereditary tree $\hhh\subset \Lambda^{<\nn}$, $\varnothing \neq D\subset 2^\Lambda$, and $\xi\in \emph{\ord}$, the following are equivalent: 
\begin{enumerate}[(i)]\item $o_D(\hhh)>\xi$, \item there exists $(x_\alpha)_{\alpha\in \ttt_\xi^D}\subset \Lambda$ so that for each $\tau\in \ttt_\xi^D$, $(x_{\tau|_i})_{i=1}^{|\tau|}\in \hhh$ and for each $t\in \ttt_\xi$ and $\sigma=(U_1, \ldots, U_n)\in D^{<\nn}$ with $n=|t|$, $x_{(t, \sigma)}\in U_n$.  \end{enumerate}

\label{lemma2}

\end{lemma}

We will call a collection $(x_\tau)_{\tau\in \ttt^D_\xi}$ satisfying the conditions of $(ii)$ a $D$ tree in $\hhh$.  

Before we begin the proof, we recall that for $(t_1,\sigma_1), (t_2, \sigma_2)\in (\Lambda_1\times \Lambda_2)^{<\nn}$, we interchangeably use $(t_1, \sigma_1)\cat (t_2, \sigma_2)$ and $(t_1\cat t_2, \sigma_1\cat \sigma_2)$ to denote the same sequence in $(\Lambda_1\times \Lambda_2)^{<\nn}$.  

\begin{proof} It is a trivial induction argument to show that if $(x_\tau)_{\tau\in \ttt_\xi^D}$ is as in $(ii)$, then for each $0\leqslant \zeta\leqslant \xi$ and each $\tau\in (\mt_\xi^D)^\zeta$, $(x_{\tau|_i})_{i=1}^{|\tau|}\in (\hhh)_D^\zeta$.  Since $o(\mt^D_\xi)= o(\mt_\xi)=\xi+1$, this means $\varnothing\in (\hhh)_D^\xi$, and $o_D(\hhh)>\xi$.   We prove the other direction by induction.  The $\xi=0$ case is trivial.  If $o_D(\hhh)>\xi+1$, then $\varnothing \in (\hhh)^{\xi+1}_D=((\hhh)_D^\xi)_D'$. This means there exists $(x_U)_{U\in D}$ so that $x_U\in U$ and $o_D(\hhh(x_U))>\xi$.  By the inductive hypothesis, we can find for each $U\in D$ some $(x^U_\tau)_{\tau\in \ttt^D_\xi}$ satisfying the conclusions with $\hhh$ replaced by $\hhh(x_U)$.  We let $$x_{(\xi+1, U)}= x_U$$ and $$x_{(\xi+1, U)\cat( t,  \sigma)}= x^U_{(t, \sigma)}$$ for each $t\in \ttt_\xi$ and $\sigma\in D^{<\nn}$ with $|t|=|\sigma|$.  Assume the result holds for each $\zeta<\xi$, $\xi$ a limit ordinal.  If $o_D(\hhh)>\xi$, then $o_D(\hhh)>\zeta+1$ for each $\zeta<\xi$.  Then by the inductive hypothesis, we find $(x^\zeta_\tau)_{\tau\in \ttt^D_{\zeta+1}}$ satisfying the conclusions of $(ii)$.  For $\tau\in \ttt^D_\xi$, there exists a unique $\zeta<\xi$ so that $\tau\in \ttt_{\zeta+1}^D$, and we let $x_\alpha= x^\zeta_\alpha$.  Clearly $(x_\tau)_{\tau\in \ttt^D_\xi}$ satisfies the requirements.  

\end{proof}

In the sequel, if $D$ is a directed set with order $\leqslant$, we will say a function $\theta:\ttt^D_\zeta\to \ttt^D_\xi$ is \emph{nice} provided $\theta$ is monotone, and if $\alpha=(t, (U_1, \ldots, U_m))$ and $\theta(\alpha)= (t_0, (V_1, \ldots, V_n))$, $U_m\leqslant V_n$.  If $\hhh$ is a hereditary tree on $\Lambda$, $D\subset 2^\Lambda$, $(x_\tau)_{\tau\in \ttt_\xi}\subset \Lambda$ is a $D$ tree in $\hhh$, and $\theta:\aaa_\zeta\to \aaa_\xi$ is nice, $(x_{\theta(\alpha)})_{\alpha \in \aaa_\zeta}$ is also a $D$ tree in $\hhh$. Recall that $\mmm$ and $(K, \delta)$ are directed sets ordered by reverse inclusion, and for any set $I$, the set of finite subsets $[I]^{<\nn}$ of $I$ is directed by inclusion.

For later use, we will prove the following lemma concerning Minkowski sums, from which the remaining part of the proof of Theorem \ref{main theorem} will follow.

\begin{lemma}  Let $X$ be a Banach space.  Let $K, L\subset X^*$ be non-empty, $w^*$ compact subsets of $X^*$.  \begin{enumerate}[(i)]\item Suppose that $s\in \hhh^K_\ee$, $0<\ee_0<\ee/4$, and $f\in s^\xi_\ee(L)$ are such that $|f(x)|<\ee_0$ for all $x\in s$.  Then $s\in (\hhh^{K+L}_{\ee_0})^\xi_w$.  \item For any $0<\delta<\ee$ and $0<\rho<\ee-\delta$, $(\hhh^L_\ee)^\xi_w \subset \hhh^{s_\delta^\xi(L)}_\ee$ and $(\hhh^L_\ee)^\xi_{(K, \delta)} \subset \hhh^{s_\rho^\xi(L)}_\ee$.  \item If $o_w(\hhh^K_{\ee_1})>\xi$ and $o_w(\hhh^L_{\ee_1})>\zeta$, then for any $0<\ee_0<\ee_1/4$, $o_w(\hhh^{K+L}_{\ee_0})>\zeta+\xi$.  

\end{enumerate}

\label{lemma coup de grace}
\end{lemma}

\begin{proof} $(i)$ Setting $\hhh= \hhh^{K+L}_{\ee_0}(s)$, by Lemma \ref{lemma2} it is sufficient to find $(x_\alpha)_{\alpha \in \aaa_\xi}\subset B_X$ so that for $\alpha=(t, (U_1, \ldots, U_n))\in \aaa_\xi$, $x_\alpha\in U_n$ and so that $s\cat(x_{\alpha|_i})_{i=1}^{|\alpha|}\in \hhh^{K+L}_{\ee_0}$. Choose $g\in K$ so that for all $x\in s$, $g(x)\geqslant \ee$.   Let $(f_\beta)_{\beta\in \bbb_\xi}\subset L$ be as in Lemma \ref{lemma1} with $f=f_\varnothing$.  Fix a sequence of positive numbers $(\ee_n)_{n=1}^\infty$ so that $ \sum_{n=0}^\infty \ee_n < \ee/4$.  Let $\mu=\sum_{n=1}^\infty \ee_n$.  We define $\varphi(\alpha)\in \bbb_\xi$ and $x_\alpha\in B_X$ for $\alpha\in \aaa_\xi$ by induction on $|\alpha|$ so that $(x_\alpha)_{\alpha\in \aaa_\xi}$ and $\varphi:\aaa_\xi\to \bbb_\xi$ satisfy for all $\alpha\in \aaa_\xi$ \begin{enumerate}[(i)]\item $|f_{\varphi(\alpha)}(x)|< \sum_{n=0}^{|\alpha|} \ee_n$ for all $x\in s$, \item $|g(x_\alpha)|< \ee/4-\sum_{n=0}^\infty \ee_n=:\delta$, \item for $1\leqslant i\leqslant |\alpha|$, $f_{\varphi(\alpha)}(x_{\alpha|_i})> \ee/4- \sum_{n=1}^{|\alpha|} \ee_n$, \item if $\alpha=(t, (U_1, \ldots, U_n))\in \aaa_\xi$, $x_\alpha\in U_n$. \item if $\alpha= (t, \sigma)$, $\varphi((t, \sigma))= (t, \sigma_0)$ for some $\sigma_0\in \nnn^{<\nn}$.      \end{enumerate}

Let us first see how this finishes the proof of $(i)$.  Assume first that $\xi>0$.  We must show that $s\cat (x_{\alpha|_i})_{i=1}^{|\alpha|}\in \hhh^{K+L}_{\ee_0}$.  Note that $g+f_{\varphi(\alpha)}\in K+L$.  For $x\in s$, $$(g+f_{\varphi(\alpha)})(x) \geqslant \ee - |f_{\varphi(\alpha)}(x)|> \ee- \sum_{n=0}^\infty \ee_n > \ee-\ee/4>\ee_0.$$  For $1\leqslant i\leqslant |\alpha|$, $$(g+f_{\varphi(\alpha)})(x_{\alpha|_i}) > \ee/4- \sum_{n=1}^{|\alpha|} \ee_n - |g(x_{\alpha|_i})|> \ee/4- \mu-\delta= \ee_0.$$   In the case that $\xi=0$, we do not need to define $\varphi$ and $x_\alpha$.  We repeat the first of these two computations above with $f_{\varphi(\alpha)}$ replaced by $f$, which finishes the proof in this case.  Therefore for the remainder we will only consider the case $\xi>0$.  

Assuming that $\varphi(\alpha)$ and $x_\alpha$ have been defined for each $\alpha\in \aaa_\xi$ with $|\alpha|<n$, we fix $\alpha\in \aaa_\xi$ with $|\alpha|=n$ (if such an $\alpha$ exists, otherwise we have already completed the definitions of $\varphi(\alpha)$ and $x_\alpha$) and define $\varphi(\alpha)$ and $x_\alpha$.  Write $\alpha= (t, \sigma\cat U)$.  If $n=1$, let $\beta=\varnothing$, and if $n>1$, let $\beta= \varphi((p(t), \sigma))= (p(t), \sigma_0)$ for some $\sigma_0\in \nnn^{<\nn}$.  Note that $(f_{(t, \sigma_0\cat V)})_{V\in \nnn}$ is a net in $L$ converging $w^*$ to $f_\beta$ so that $\|f_\beta - f_{(t, \sigma_0\cat V)}\|> \ee/2$ for all $V\in \nnn$.   For all $V\in \nnn$, let $y_V\in B_X$ be chosen so that $(f_{(t, \sigma_0\cat V)} - f_\beta)(y_V)>\ee/2$.   By passing to a subnet $(y_V)_{V\in D}$ of $(y_V)_{V\in \nnn}$, we may assume that for all $V_1, V_2\in D$, $$y_{V_2}-y_{V_1}\in U\cap \{x\in X: |g(x)|<\delta\}\cap \{x\in X: |f_\beta(x)|<\ee_n\}.$$  Let $V_1\in D$ be fixed and choose $V_2\in D$ so that $|(f_{(t, \sigma_0\cat {V_2})}- f_\beta)(y_{V_1})|< \ee_n$, for each $1\leqslant i<n$, $|(f_{(t, \sigma_0\cat V_2)}-f_\beta)(x_{\alpha|_i})|<\ee_n$, and for each $x\in s$, $|(f_{(t, \sigma_0\cat V_2)}-f_\beta)(x)|<\ee_n$.  Of course, we may do this since $(f_{(t_0, \sigma_0\cat V)})_{V\in D}$ converges $w^*$ to $f_\beta$.  Define $x_\alpha=(y_{V_2}-y_{V_1})/2$ and $\varphi(\alpha)= f_{(t, \sigma_0\cat V_2)}$, and note that (ii), (iv) and (v) are satisfied by this construction.       

Fix $x\in s$.  If $n=1$, $$|f_{\varphi(\alpha)}(x)| \leqslant |f_\varnothing(x)|+ |(f_{\varphi(\alpha)}-f_\varnothing)(x)|< \ee_0+\ee_1.$$ If $n>1$, $$|f_{\varphi(\alpha)}(x)|\leqslant |(f_{\varphi(\alpha)}-f_\beta)(x)|+ |f_\beta(x)|  < \ee_n+ \sum_{i=0}^{n-1} \ee_i =\sum_{i=0}^n \ee_i.$$  This shows that (i) is satisfied.  

For $1\leqslant i<n$, $$f_{\varphi(\alpha)}(x_{\alpha|_i}) \geqslant f_\beta(x_{\alpha|_i}) - |(f_{\varphi(\alpha)}-f_\beta)(x_{\alpha|_i})| > \ee/4- \sum_{i=1}^{n-1}\ee_i - \ee_n = \ee/4- \sum_{i=1}^n \ee_i.$$  Also, \begin{align*} f_{\varphi(\alpha)}(x_\alpha) & \geqslant \frac{1}{2}(f_{\varphi(\alpha)}-f_\beta)(y_{V_2}) - \frac{1}{2}|(f_{\varphi(\alpha)}- f_\beta)(y_{V_1})| -|f_\beta(x_\alpha)| \\ & > \frac{1}{2}(\ee/2) - \ee_n/2-\ee_n/2 = \ee/4 - \ee_n.\end{align*} This shows (iii), and this completes the proof of $(i)$.

$(ii)$ We prove both containments simultaneously by induction.  For $\xi=0$, we have equality by definition.  Next, assume the result holds for a given $\xi$.  If $(\hhh^L_\ee)_w^{\xi+1}=\varnothing$, of course $(\hhh^L_\ee)_w^{\xi+1}\subset \hhh^{s_\delta^{\xi+1}(L)}_\ee$. Otherwise fix $s\in (\hhh^L_\ee)_w^{\xi+1}$ and choose $(x_U)_{U\in \mmm}\subset B_X$ so that $s\cat x_U\in (\hhh^L_\ee)_w^\xi$ for all $U\in \mmm$.  For each $U$, by the inductive hypothesis we may select $f_U\in s_\delta^\xi(L)$ so that $f_U(x)\geqslant \ee$ for all $x\in s$ and so that $f_U(x_U)\geqslant \ee$.  Let $f$ be the $w^*$ limit of a $w^*$ converging subnet $(f_U)_{U\in D}$ of $(f_U)_{U\in \mmm}$. By $w^*$ compactness of $s_\delta^\xi(L)$, $f\in s_\delta^\xi(L)$.  Moreover $$\underset{U\in D}{\lim \sup} \|f-f_U\| \geqslant \underset{U\in D}{\lim \sup}(f_U-f)(x_U) = \underset{U\in D}{\lim \sup} f_U(x_U) \geqslant \ee> \delta.$$  This means $f\in s_\delta^{\xi+1}(L)$, and of course $f(x)\geqslant \ee$ for all $x\in s$.  This gives the successor case of the first inclusion.  The second inclusion is similar, except we replace $\mmm$ by $(K, \delta)$ and let $(f_U)_{U\in D}$ be a $w^*$ converging subnet of $(f_U)_{U\in (K, \delta)}$.  Then if $f$ is the $w^*$ limit of this subnet, $$\underset{U\in D}{\lim \sup} \|f-f_U\| \geqslant \underset{U\in D}{\lim \sup}(f_U-f)(x_U) \geqslant  \underset{U\in D}{\lim \sup} f_U(x_U) - \delta \geqslant \ee-\delta >\rho,$$ and $f\in s_\rho^{\xi+1}(L)$.  Here we have used the fact that for any subnet $(x_U)_{U\in D}$ of $(x_U)_{U\in (K, \delta)}$ and any $g\in K$, $\underset{U\in D}{\lim \sup}|g(x_U)| \leqslant \delta$ by the definition of $(K, \delta)$.  

Assume the result holds for all $\zeta<\xi$, $\xi$ a limit ordinal.  If $(\hhh^L_\ee)_w^\xi=\varnothing$, of course $(\hhh^L_\ee)_w^\xi \subset \hhh_\ee^{s_\ee^\xi(L)}$.  Otherwise fix $s\in (\hhh^L_\ee)_w^\xi$.  This means $s\in (\hhh^L_\ee)_w^\zeta$ for each $\zeta<\xi$, whence there exists $(f_\zeta)_{\zeta<\xi}$ so that $f_\zeta\in s_\delta^\zeta(L)$ and $f_\zeta(x)\geqslant \ee$ for all $x\in s$ and $\zeta<\xi$.  If $f$ is any $w^*$ limit of a $w^*$ converging subnet of $(f_\zeta)_{\zeta<\xi}$, we deduce $f\in s^\xi_\ee(L)$ and $f(x)\geqslant \ee$ for all $x\in s$, giving the limit ordinal case of the first inclusion.  The second inclusion is similar, with $\mmm$ replaced by $(K, \delta)$.

$(iii)$ Fix $0<\ee<\ee_1$ so that $0<\ee_0<\ee/4$.  By $(ii)$, $ s_\ee^\zeta(L)\neq \varnothing$.  Fix $f\in s_\ee^\zeta(L)$.  Since $o_w(\hhh^K_\ee)>\xi$, by  Lemma \ref{lemma2}, we may choose $(x_\alpha)_{\alpha\in \aaa_\xi}$ so that if $\alpha=(t, (U_1, \ldots, U_n))\in \aaa_\xi$, $x_\alpha\in U_n$ and for each $\alpha\in \aaa_\xi$, $(x_{\alpha|_i})_{i=1}^{|\alpha|}\in \hhh^K_\ee$. Let $V=\{x\in X: |f(x)|<\ee_0\}$.  By replacing $x_{(t, (U_1, \ldots, U_n))}$ with $x_{(t, (V\cap U_1, \ldots, V\cap U_n))}$, we may assume that for each $\alpha\in \aaa_\xi$, $|f(x_\alpha)|< \ee_0$.  Then by $(i)$, $(x_{\alpha|_i})_{i=1}^{|\alpha|}\in (\hhh^{K+L}_{\ee_0})_w^\zeta$.  But appealing again to Lemma \ref{lemma2}, $o_w((\hhh^{K+L}_{\ee_0})_w^\zeta)>\xi$, whence $o_w(\hhh^{K+L}_{\ee_0})>\zeta+\xi$.

\end{proof}

\begin{proof}[Proof of Theorem \ref{main theorem}] We have already argued that $(ii)\Rightarrow (iii)$ after the statement of the theorem.  The implication $(i)\Rightarrow (ii)$ follows from Lemma \ref{lemma coup de grace}$(i)$ with $s=\varnothing$ and $K=\{0\}$.  Then $s^\xi_\ee(L)\neq \varnothing$ implies $o_w(\hhh^L_{\ee_0})>\xi$ for any $0<\ee_0<\ee/4$.  The implication $(iii)\Rightarrow (i)$ follows from the second inclusion of Lemma \ref{lemma coup de grace}$(ii)$.

\end{proof}

\section{Sum estimate applications} The remainder of this note is devoted to applications of Theorem \ref{main theorem}. The first section of applications deals with results yielding sum estimates, which are naturally grouped together and deduced as consequences of related coloring lemmas which we discuss at the end of this section.

The following facts about ordinals can be found in \cite{M}.   Recall that any ordinal $\xi$ can be uniquely written as $$\xi=\omega^{\alpha_1}n_1+\ldots +\omega^{\alpha_k}n_k,$$ where $n_i\in \nn$, $\alpha_1>\ldots >\alpha_k$, and $k=0$ if $\xi=0$.  Here, $\omega$ denotes the first infinite ordinal.  This representation is the \emph{Cantor normal form}.  If $\xi, \zeta$ are two ordinals, by allowing either $m_i=0$ or $n_i=0$, we can write $\xi=\omega^{\alpha_1}m_1+\ldots +\omega^{\alpha_k} m_k$, $\zeta= \omega^{\alpha_1}n_1+\ldots + \omega^{\alpha_k} n_k$.  Then the \emph{Hessenberg} (or \emph{natural}) \emph{sum} of $\xi$ and $\zeta$, denoted $\xi\oplus\zeta$, is defined to be $\omega^{\alpha_1}(m_1+n_1)+\ldots + \omega^{\alpha_k} (m_k+n_k)$.  Note that this is well-defined, as non-uniqueness of representation only yields extraneous zero terms in the sum. We remark that for any fixed $\delta\in \ord$, $\gamma\mapsto \delta\oplus \gamma$ is strictly increasing.    To save writing, we will agree that $\infty\oplus \xi=\xi\oplus \infty=\infty$ for any $\xi\in\ord\cup \{\infty\}$.  

We recall the definition of gamma and delta numbers.  An ordinal $\xi$ is called a gamma number if for any $\zeta, \eta<\xi$, $\zeta+\eta<\xi$. The ordinal $\xi$ is a gamma number if and only if for any $\zeta<\xi$, $\zeta+\xi=\xi$, or equivalently, $\xi=0$ or $\xi=\omega^\eta$ for some $\eta\in \ord$.  An ordinal $\xi$ is called a delta number if for any $\zeta, \eta<\xi$, $\zeta\eta<\xi$.  The ordinal $\xi$ is a delta number if and only if for any $0<\zeta<\xi$, $\zeta\xi=\xi$, or equivalently, $\xi=0$, $\xi=1$, or $\xi=\omega^{\omega^\zeta}$ for some $\zeta\in \ord$.

 Throughout, a $K$-unconditional basis $(e_i)_{i\in I}$ for the Banach space $E$ will be an unordered subset of $E$ having dense span in $E$ so that for every pair of finite subsets $J_1, J_2$ of $I$, all scalars $(a_n)_{n\in J_1\cup J_2}$, and all scalars $(\ee_n)_{n\in J_1\cup J_2}$ so that $|\ee_n|=1$ for each $n\in J_1\cup J_2$, $$\sum_{n\in J_1\cup J_2} a_n e_n\mapsto \sum_{n\in J_1}a_ne_n- \sum_{n\in J_2} a_ne_n$$ is a well-defined, continuous projection of norm not more than $K$.   In this case, every $x\in E$ has a unique representation $x=\sum_{n\in I} a_ne_n$, where $\{n\in I: a_n\neq 0\}$ is countable and the series $\sum_{n\in I}a_ne_n$ converges unconditionally to $x$.  Moreover, for every $J\subset I$ and any $(\ee_n)_{n\in J}$ with $|\ee_n|\leqslant 1$ for all $n\in I$, the map $\sum_{n\in I}a_ne_n\mapsto \sum_{n\in I}\ee_na_n e_n$ is well-defined with norm not exceeding $K$.  We can always equivalently renorm a Banach space with a $K$-unconditional basis $(e_i)_{i\in I}$ so that $(e_i)_{i\in I}$ becomes a $1$-unconditional basis for $E$ with the new norm. If we are not concerned with the constant $K$, we will simply say $(e_i)_{i\in I}$ is an unconditional basis for $E$. We let $(e_i^*)_{i\in I}$ denote the biorthogonal functionals to $E$, which is a $K$-unconditional basis for its closed span.  It is well-known that a Banach space $E$ with an unconditional basis $(e_i)_{i\in I}$ must contain an isomorphic copy of $\ell_1$, or the closed span of the coordinate functionals $(e_i^*)_{i\in I}$ is all of $E^*$.  Similarly, if $E, F$ are Banach spaces with unconditional bases $(e_i)_{i\in I}$ and $(f_i)_{i\in I}$, respectively, and $B:E\to F$ is a diagonal operator (meaning that $Be_i=b_if_i$ for some scalars $(b_i)_{i\in I}$), then either $B$ preserves an isomorphic copy of $\ell_1$, or $B^*F^*\subset [e_i^*:i\in I]$.  We remark that $E^*$ can be naturally identified with the set of all formal (not necessarily countably non-zero or norm converging) series $\sum_{i\in I} a_i e_i^*$ so that $\sup_{J\in [I]^{<\nn}} \|\sum_{i\in J}a_ie_i^*\|<\infty$, and $\|\sum_{i\in I}a_ie_i^*\|= \sup_{J\in [I]^{<\nn}} \|\sum_{i\in J}a_ie_i^*\|$.   

If $(e_i)_{i\in I}$ is a $1$-unconditional basis for the Banach space $E$, and if $(X_i)_{i\in I}$ is a collection of Banach spaces, the direct sum $\bigl(\oplus_{i\in I}X_i\bigr)_E$ is the set of all tuples $(x_i)_{i\in I}$ so that $x_i\in X_i$ and $\sum_{i\in I} \|x_i\|e_i\in E$.  We note that $\bigl(\oplus_{i\in I}X_i\bigr)_E$ is a Banach space when endowed with the norm $\|(x_i)_{i\in I}\|=\|\sum_{i\in I}\|x_i\|e_i\|$.  In this case, $(\oplus_{i\in I}X_i)_E^*$ can be naturally isometrically identified with all tuples $(x_i^*)_{i\in I}\in \prod_{i\in I}X_i^*$ so that the formal series $\sum_{i\in I}\|x_i^*\|e_i^*$ lies in $E^*$.

\subsection{Estimates for Minkowski sums}

The main result of this subsection is the following.  

\begin{theorem} Let $K, L\subset X^*$ $w^*$ compact and non-empty.  \begin{enumerate}[(i)]\item $Sz(K\cup L)=\max\{Sz(K), Sz(L)\}$,\item  $\sup_{\ee>0}( Sz_\ee(K)+Sz_\ee(L)) \leqslant Sz(K+L)$.  \item There exists a positive constant $C$ so that for all $\ee>0$, $Sz_\ee(K+L)\leqslant Sz_{\ee/C}(K)\oplus Sz_{\ee/C}(L)$. \item  If $K$ and $L$ are convex, then $Sz(K+L)=\max\{Sz(K), Sz(L)\}$.  \end{enumerate} 

\label{Minkowski}

\end{theorem}

For this, we will need the following concerning what values may be attained by the Szlenk index of a convex set.  

\begin{proposition} Let $\varnothing \neq K$ be a $w^*$ compact subset of $X^*$.  \begin{enumerate}[(i)]\item $o_w(\hhh^K_\ee)=1$ for every $\ee>0$ if and only if $K$ is norm compact.   \item If $K$ is convex, $0<\delta < \ee/8$, and $\zeta\in \emph{\ord}$ is such that $\zeta<o_w(\hhh^K_\ee)$, then $\zeta \cdot 2 < o_w(\hhh^K_\delta)$.  \item Either $\sup_{\ee>0}o_w(\hhh^K_\ee)=\infty$ or there exists $\xi\in \emph{\ord}$ so that $\sup_{\ee>0}o_w(\hhh^K_\ee)=\omega^\xi$, and this supremum is attained if and only if $K$ is norm compact.  \end{enumerate}

\label{special form}
\end{proposition}

\begin{remark} Item $(iii)$ of Theorem \ref{Minkowski} cannot be non-trivially deduced from results appearing in the literature.  

Part $(iii)$ of Proposition \ref{special form} was shown in \cite{AJO} in the case that $K=B_{X^*}$ where $X$ is a Banach space having separable dual.  We note that the proof given here is not a modification of that proof, which depended on the separability of $X$ and $X^*$.

We next note the origins of some of these results which appear in the literature or which can be deduced from results appearing in the literature which use the slicing definition of the Szlenk index.  Item $(i)$ of Theorem \ref{Minkowski} as well as item $(i)$ of Proposition \ref{special form} were shown by Brooker \cite{B}.  The first part of item $(iii)$ of Proposition \ref{special form} in the case that $K=B_{X^*}$ was shown using the slicing definition by Lancien \cite{L0}, and one can see that the proof applies to any non-empty, $w^*$ compact, convex set $K$.  More generally, this method can be seen to imply $(ii)$ of Theorem \ref{Minkowski}.  Item $(iii)$ in the case that $K=B_{X^*}$ and $L=B_{Y^*}$ for separable Banach spaces $X$ and $Y$, and that $K+L\subset (X\oplus Y)^*$ was treated in \cite{OSZ}.

\end{remark}

\begin{proof}$(i)$ Assume $o_w(\hhh^K_\ee)>1$.  This means $\varnothing\in (\hhh^K_\ee)'_w$, and there exists $(x_U)_{U\in \mmm}\subset B_X$ so that $x_U\in U$ and $(x_U)\in \hhh^K_\ee$ for all $U\in \mmm$.  Choose $(x^*_U)_{U\in \mmm}\subset K$ so that $x^*_U(x_U)\geqslant\ee$. By norm compactness, we may pass to a norm converging subnet $(x_U^*)_{U\in D}$ and note that if $\lim_{U\in D}x^*_U=x^*$, $$\lim_{U\in D} x^*_U(x_U)=\lim_{U\in D} x^*(x_U)=0,$$ since $(x^*_U)_{U\in D}$ is a weakly null net.  This contradiction implies that if $K$ is norm compact, $o_w(K_\ee)=1$ for every $\ee>0$.  Next, if $K$ is not norm compact, we may choose $\ee>0$ and an infinite subset $S$ of $K$ so that if $x^*_1, x^*_2\in S$ are distinct, $\|x^*_1-x^*_2\|> 4\ee$.  We may choose $x^*\in K$ which fails to be $w^*$ isolated in $S$ and, by replacing $S$ with $S\setminus \{x^*\}$, we may assume $x^*\notin S$.  We may choose a net $(x^*_\lambda)_{\lambda\in D}$ in $S$ converging $w^*$ to $x^*$ and, for each $\lambda\in D$, we may choose $x_\lambda\in B_X$ so that $(x^*_\lambda-x^*)(x_\lambda)>4\ee$.  Choose $U\in \mmm$ and, by passing to a subnet of $(x_\lambda)_{\lambda\in D}$ and the corresponding subnet of $(x_\lambda^*)$, we assume that for each $\lambda_1, \lambda_2\in D$, $x_{\lambda_2}-x_{\lambda_1}\in U$ and $|x^*(x_{\lambda_1}-x_{\lambda_2})|<\ee$. Fix $\lambda_1\in D$. Then $$\lim\sup_\lambda x^*_\lambda(x_\lambda- x_{\lambda_1})\geqslant \lim\sup_\lambda (x^*_\lambda-x^*)(x_\lambda-x_{\lambda_1})-\ee \geqslant 3\ee- \lim_\lambda (x^*_\lambda-x^*)(x_{\lambda_1})=3\ee.$$  Then by taking $x_U=(x_\lambda-x_{\lambda_1})/2$ for some $\lambda$, we can guarantee $x^*_\lambda(x_U)> \ee$.  Then the net $(x_U)_{U\in \mmm}$ witnesses the fact that $\varnothing \in (\hhh^K_\ee)_w'$ and $o_w(\hhh^K_\ee)>1$.

$(ii)$  Note that $K=\frac{1}{2}K+ \frac{1}{2}K$. It is obvious that $\hhh^{\frac{1}{2}K}_{\ee/2}= \hhh^K_\ee$, so $\zeta<o_w(\hhh^{\frac{1}{2}K}_{\ee/2})$. By Lemma \ref{lemma coup de grace}$(iii)$ with $K$ and $L$ replaced by $\frac{1}{2}K$ and $\ee_1$ replaced by $\ee/2$, $\zeta\cdot 2 < o_w(\hhh^K_\delta)$.

$(iii)$ Assume $\sup_{\ee>0}o_w(\hhh^K_\ee)<\infty$.  If $\zeta<\sup_{\ee>0}o_w(\hhh^K_\ee)$, we may choose $\ee>0$ with $\zeta< o_w(\hhh^K_\ee)$.  Then $\zeta \cdot 2<o_w(\hhh^K_{\ee/9})$.  In particular, $0<\sup_{\ee>0} o_w(\hhh^K_\ee)$ and $\sup_{\ee>0} o_w(\hhh^K_\ee)$ is a gamma number.  This means $\sup_{\ee>0} o_w(\hhh^K_\ee)=\omega^\xi$ for some ordinal $\xi$, since this supremum cannot be zero.  If $0<\zeta<o_w(\hhh^K_\ee)$, then $\zeta< \zeta \cdot 2 < o_w(\hhh^K_{\ee/9})$.  Since by $(i)$ such a $\zeta$ exists if and only if $K$ fails to be norm compact, we deduce that the supremum is attained if and only if $K$ is norm compact.

\end{proof}

If $(x_i)_{i=1}^n$ is any sequence in the Banach space $X$ and $f,g\in X^*$, are such that $(f+g)(x_i)\geqslant \ee$ for each $1\leqslant i\leqslant n$, then of course there exist $p,q\in \nn_0$ with $p+q=n$ and subsets $A,B$ of $\{1, \ldots,n\}$ with $|A|=p$, $|B|=q$, $f(x_i)\geqslant \ee/2$ and $g(x_j)\geqslant \ee/2$ for all $i\in A$ and $j\in B$.   We will perform a transfinite version of this argument, which will yield most of Theorem \ref{Minkowski} as an easy consequence.  Namely, we will show that if $o_w(\hhh^{K+L}_\ee)>\xi$, there exist ordinals $\eta, \zeta$ with $\eta\oplus \zeta=\xi$ so that $o_w(\hhh^K_{\ee/2})>\eta$ and $o_w(\hhh^L_{\ee/2})>\zeta$.  The execution of this argument is somewhat technical, and similar to the analogous result appearing in \cite{C2} where the family $\aaa_\xi$ was replaced by the fine Schreier family $\fff_\xi$ in the case that $\xi$ is countable.  For this reason, we will omit the details which follow unaltered from the argument appearing there.  

For $\zeta, \xi\in \ord$, if $\theta:\aaa_\zeta\to \aaa_\xi$ and $e:MAX(\aaa_\zeta)\to MAX(\aaa_\xi)$ are any functions, we say the pair $(\theta, e)$ is \emph{extremely nice} provided \begin{enumerate}[(i)]\item $\theta$ is nice, \item for each $\alpha\in MAX(\aaa_\zeta)$, $\theta(\alpha)\preceq e(\alpha)$. \end{enumerate} By an abuse of notation, we write $(\theta, e):\aaa_\zeta\to \aaa_\xi$ rather than $(\theta, e):\aaa_\zeta\times MAX(\aaa_\zeta)\to \aaa_\xi\times MAX(\aaa_\xi)$.   It is easy to see that if $(\theta_1, e_1):\aaa_\zeta\to \aaa_\eta$ and $(\theta_2, e_2):\aaa_\eta\to \aaa_\xi$ are extremely nice, then $(\theta_2\circ\theta_1, e_2\circ e_1)$ is extremely nice. Moreover, we consider the empty map from $\aaa_0$ into $\aaa_\xi$ to be extremely nice for any $\xi$.   

\begin{lemma} \begin{enumerate}[(i)]\item If $\zeta\leqslant \xi$, $\zeta, \xi\in \emph{\ord}$, there exists an extremely nice $(\theta, e):\aaa_\zeta\to \aaa_\xi$.  \item For $\xi\in \emph{\ord}$, if $\ccc^1, \ccc^2\subset MAX(\aaa_\xi)$, then there exists an extremely nice $(\theta, e):\aaa_\xi\to \aaa_\xi$ and $j\in \{1,2\}$ so that $e(MAX(\aaa_\xi))\subset \ccc^j$.  \item If $K_1, K_2\subset X^*$, $(x_\alpha)_{\alpha\in \aaa_\xi}$, $(f_\alpha^1)_{\alpha\in MAX(\aaa_\xi)}\subset K_1$, $(f_\alpha^2)_{\alpha\in MAX(\aaa_\xi)}\subset K_2$, and $\ee>0$ are such that for each $\alpha\in MAX(\aaa_\xi)$ and $1\leqslant i\leqslant |\alpha|$,  $(f^1_\alpha+f^2_\alpha)(x_{\alpha|_i})\geqslant \ee$, there exist $\zeta_1, \zeta_2\in \emph{\ord}$ with $\zeta_1\oplus \zeta_2=\xi$ and for $j\in \{1,2\}$, extremely nice $(\theta_j, e_j):\aaa_{\zeta_j}\to \aaa_\xi$ so that for each $\alpha\in MAX(\aaa_{\zeta_j})$ and each $1\leqslant i\leqslant |\alpha|$, $f^j_{e_j(\alpha)}(x_{\theta_j(\alpha|_i)})\geqslant \ee/2$.   \end{enumerate}

\label{useful lemma}

\end{lemma}

\begin{proof}$(i)$ By \cite{C}, there exists $\varphi:\ttt_\zeta\to \ttt_\xi$ which is monotone and $|t|=|\varphi(t)|$ for all $t\in \ttt_\zeta$.  Then define $\theta:\aaa_\zeta\to \aaa_\xi$ by letting $\theta((t, \sigma))= (\varphi(t), \sigma)$.  It is clear that $\theta$ is nice.   Since $\aaa_\xi$ is well-founded, for each $\alpha\in \aaa_\zeta$, there exists some $\beta\in MAX(\aaa_\xi)$ extending $\theta(\alpha)$.  Let $e(\alpha)=\beta$.  Then $(\theta, e)$ is extremely nice.

$(ii)$ We prove the result by induction.  The $\xi=0$ case is vacuous. Suppose $\ccc^1\cup \ccc^2=MAX(\aaa_1)=\aaa_1=\{(1, U):U\in \mmm\}$.  Choose $\phi:\mmm\to \mmm$ and $j\in \{1,2\}$ so that $\phi(U)\subset U$ and $(1, \phi(U))\in \ccc^j$.  Let $\theta((1,U))=e((1, U))= (1, \phi(U))$.

 Assume the result holds for a given $\xi>0$ and $\ccc^1\cup \ccc^2=MAX(\aaa_{\xi+1})$. For each $U\in \mmm$ and $j\in \{1,2\}$, let $$\ccc^j(U)=\{\alpha\in MAX(\aaa_\xi): (\xi+1, U)\cat \alpha\in \ccc^j\}.$$  Note that for each $U\in \mmm$, $\ccc^1(U)\cup \ccc^2(U)=MAX(\aaa_\xi)$.  For each $U$, choose $j_U\in \{1, 2\}$,  $(\theta_U, e_U):\aaa_\xi \to \aaa_\xi$ extremely nice so that $e_U(MAX(\aaa_\xi))\subset \ccc^{j_U}(U)$.  Choose $\phi:\mmm\to \mmm$  and $j\in \{1,2\}$ so that for all $U\in \mmm$, $U\supset \phi(U)$ and $j_{\phi(U)}=j$.  Define the extremely nice $(\theta, e)$ by letting $$\theta(\xi+1, U)= (\xi+1, \phi(U)),$$ and for $t\in \ttt_\xi$, $\sigma\in \mmm^{<\nn}$ with $|t|=|\sigma|$, let  $$\theta((\xi+1, U)\cat (t, \sigma))= (\xi+1, \phi(U))\cat \theta_{\phi(U)}(t, \sigma).$$    If $\xi=0$, we let $$e(1, U)=(1, \phi(U)).$$  If $\xi>0$, $(\xi+1,U)\cat ( t, \sigma)\in MAX(\aaa_{\xi+1})$ only if $(t,\sigma)\in MAX(\aaa_\xi)$, and we let $$e((\xi+1, U)\cat (t, \sigma))=(\xi+1, \phi(U))\cat e_{\phi(U)}(t, \sigma).$$  

Assume the result holds for all $\zeta<\xi$, $\xi$ a limit ordinal.  Assume $\ccc^1\cup \ccc^2=MAX(\aaa_\xi)$.  For each $\zeta<\xi$ and $j\in \{1,2\}$, let $$\ccc^j(\zeta)=\ccc^j\cap MAX(\aaa_{\zeta+1}).$$  Then $\ccc^1(\zeta)\cup \ccc^2(\zeta)=MAX(\aaa_{\zeta+1})$.  For each $\zeta<\xi$, choose $j_\zeta\in \{1,2\}$ and an extremely nice $(\theta_\zeta, e_\zeta):\aaa_{\zeta+1}\to \aaa_{\zeta+1}$ so that $e_\zeta(\max(\aaa_{\zeta+1}))\subset \ccc^{j_\zeta}(\zeta)$.  Choose $j\in \{1,2\}$ and $\phi:[0,\xi)\to [0,\xi)$ so that for each $\zeta<\xi$, $\zeta\leqslant \phi(\zeta)$ and $j_{\phi(\zeta)}=j$.  By $(i)$, we may choose for each $\zeta<\xi$ some extremely nice $(\theta'_\zeta, e'_\zeta):\aaa_{\zeta+1}\to \aaa_{\phi(\zeta)+1}$.  Then define $\theta$ on $\aaa_\xi$ by letting $\theta|_{\ttt_{\zeta+1}}=\theta_{\phi(\zeta)}\circ \theta'_\zeta$ and define $e$ on $MAX(\aaa_\xi)$ by letting $e|_{MAX(\aaa_{\zeta+1})}= e_{\phi(\zeta)}\circ e'_\zeta$.

$(iii)$ By induction on $\xi$.  The $\xi=0$ case is trivial.  Assume the assertion holds for a given $\xi$ and $(x_\alpha)_{\alpha\in \aaa_{\xi+1}}$, $(f^1_\alpha)_{\alpha\in MAX(\aaa_{\xi+1})}\subset K_1$, $(f^2_\alpha)_{\alpha\in MAX(\aaa_{\xi+1})}\subset K_2$, and $\ee>0$  are as in the statement of $(iii)$. We first claim that we may assume without loss of generality that there exists $k\in \{1,2\}$ so that for each $U\in \mmm$ and each $\alpha\in MAX(\aaa_{\xi+1})$ with $(\xi+1, U)\preceq \alpha$, $f^k_\alpha(x_{(\xi+1, U)})\geqslant \ee/2$.      This is because if we let $\ccc^j$ consist of those $\alpha=(t, (U_1, \ldots, U_n))\in MAX(\aaa_{\xi+1})$ so that $f^j_\alpha(x_{(\xi+1, U_1)})\geqslant \ee/2$, $\ccc^1\cup \ccc^2=MAX(\aaa_{\xi+1})$.  We may then find by $(ii)$ some $k\in \{1,2\}$ and an extremely nice $(\theta, e)$ so that $e(MAX(\aaa_{\xi+1}))\subset \ccc^k$.  Then if we replace $x_\alpha$ by $x_{\theta(\alpha)}$, $f^1_\alpha$ by $f^1_{e(\alpha)}$, and $f^2_\alpha$ by $f^2_{e(\alpha)}$, the resulting collections still satisfy the hypotheses of $(iii)$ and have the additional property.  We therefore assume that $(x_\alpha)_{\alpha\in \aaa_\xi}\subset B_X$, $(f^1_\alpha)_{\alpha\in MAX(\aaa_{\xi+1})}\subset K_1$, and $(f^2_\alpha)_{\alpha\in MAX(\aaa_{\xi+1})}\subset K_2$ are as in the statement of $(iii)$ and that, without loss of generality $k=1$, so that for each $U\in \mmm$ and each $\alpha\in MAX(\aaa_{\xi+1})$ with $(\xi+1, U)\preceq \alpha$, $f^1_\alpha(x_{(\xi+1, U)})\geqslant \ee/2$. If $\xi=0$, we let $\zeta_1=1$, $\zeta_0=0$, $\theta_1(1,U)=e_1(1,U)=(1,U)$ and $\theta_2$, $e_2$ be the empty maps.  One easily checks that this completes the case $\xi+1=1$.  In the case that $\xi>0$, for each $U\in \mmm$ and $(t, \sigma)\in \aaa_\xi$, let $x_{(t, \sigma)}(U)=x_{(\xi+1, U)\cat (t, \sigma)}$.  For $j\in \{1, 2\}$ and $(t, \sigma)\in MAX(\aaa_\xi)$, let $f^j_{(t, \sigma)}(U)=f^j_{(\xi+1, U)\cat (t, \sigma)}$.   Then for each $U\in \mmm$, $(x_\alpha(U))_{\alpha\in \aaa_\xi}$, $(f^1_\alpha(U))_{\alpha\in MAX(\aaa_\xi)}$, and $(f^2_\alpha(U))_{\alpha\in MAX(\aaa_\xi)}$ satisfy the conditions required to apply the inductive hypothesis.  For $U\in \mmm$ and $j\in \{1,2\}$, there exist ordinals $\zeta_j(U)$ and extremely nice $(\theta_j^U, e_j^U):\aaa_{\zeta_j}\to \aaa_\xi$ satisfying the conclusions.  Since there are only finitely many pairs $\zeta_1, \zeta_2$ with $\zeta_1\oplus \zeta_2=\xi$, we may choose $\phi:\mmm\to \mmm$ so that $\phi(U)\subset U$ for all $U\in \mmm$ and ordinals $\zeta_1, \zeta_2$ with $\zeta_1=\zeta_1(\phi(U))$ and $\zeta_2=\zeta_2(\phi(U))$ for all $U\in \mmm$.  By replacing $x_{(t, U\cat \sigma)}$ by $x_{(t, \phi(U)\cat \sigma)}$ for each $(t, U\cat\sigma)\in \aaa_{\xi+1}$ and replacing $f^j_{(t, U\cat \sigma)}$ by $f^j_{(t, \phi(U)\cat\sigma)}$ for $j=1,2$ and all $t\in MAX(\aaa_{\xi+1})$, $\theta^U$ by $\theta^{\phi(U)}$, etc., we may assume that $\zeta_1(U)=\zeta_1$ and $\zeta_2(U)=\zeta_2$ for all $U\in \mmm$. If $\zeta_2=0$, we take $(\theta_2, e_2)$ to be the empty map.  Otherwise fix $V\in \mmm$ and let $(\theta_2, e_2):\aaa_{\zeta_2}\to \aaa_{\xi+1}$ be defined by $\theta_2(t, \sigma)=(\xi+1, V)\cat \theta_2^V(t, \sigma)$.  We similarly define $e_2$ by $e_2((t, \sigma))=(\xi+1, V)\cat e_2^V(t, \sigma)$.  If $\zeta_1=0$, we define $(\theta_1, e_1):\aaa_1\to \aaa_{\xi+1}$ by letting $\theta_1(1, U)=e_1(1,U)=(1,U)$.  If $\zeta_1>0$, we define $(\theta_1, e_1):\aaa_{\zeta_1+1} \to \aaa_{\xi+1}$ by letting $\theta_1(\zeta_1+1, U)= (\xi+1, U)$, $\theta_1((\zeta_1+1, U)\cat (t, \sigma))= (\xi+1, U)\cat \theta^U_1(t, \sigma)$.   It is straightforward to check that these maps are all well-defined and satisfy the conclusions.

Assume the result holds for every $\zeta<\xi$, $\xi$ a limit ordinal.  Assume $(x_\alpha)_{\alpha\in \aaa_\xi}$, $(f^j_\alpha)_{\alpha\in MAX(\aaa_\xi)}$ are as in the statement of $(iii)$.  For each $\eta<\xi$, apply the inductive hypothesis to $(x_\alpha)_{\alpha\in \aaa_{\eta+1}}$, $(f^j_\alpha)_{\alpha\in MAX(\aaa_{\eta+1})}$ to obtain $\zeta_1(\eta)$, $\zeta_2(\eta)$ with $\zeta_1(\eta)\oplus \zeta_2(\eta)=\eta+1$ and extremely nice $(\theta^\zeta_j, e^\zeta_j):\aaa_{\zeta_1(\eta)}\to \aaa_{\eta+1}\subset \aaa_\xi$ satisfying the conclusions.  By \cite{C}, there exist a subset $M\subset [0, \xi)$ and ordinals $\gamma$, $\delta$, and $(\gamma_\eta)_{\eta\in M}$ so that (after switching $K_1$ and $K_2$ if necessary) \begin{enumerate}[(i)]\item for each $\eta\in M$, $\zeta_2(\eta) \geqslant \delta$, \item for each $\eta\in M$, $\gamma+\gamma_\eta=\zeta_1(\eta)$, \item $(\gamma+\sup_{\eta\in M}\gamma_\eta)\oplus \delta=\xi$. \end{enumerate} 

Let $\gamma'=\sup_{\eta\in M}\gamma_\eta$. Note that property (iii) implies $\gamma+\gamma'$ and $\delta$ are both limit ordinals. We will define $(\theta_1, e_1):\aaa_{\gamma+\gamma'}\to \aaa_\xi$ and $(\theta_2, e_2):\aaa_\delta \to \aaa_\xi$ to satisfy the conclusions.  Since $(\gamma+\gamma')\oplus \delta=\xi$, this will finish the proof.  Choose any $\eta\in M$ and any extremely nice $(\theta', e'):\aaa_\delta\to \aaa_{\zeta_2(\eta)}$ and let $(\theta_2, e_2)= (\theta_{\zeta_2(\eta)} \circ \theta', e_{\zeta_2(\eta)}\circ e')$.

Choose $\phi:[0, \gamma+\gamma')\to M$ so that for each $\eta<\gamma+\gamma'$, $\eta+1\leqslant \gamma+\gamma'_{\phi(\eta)}=\zeta_1(\phi(\eta))$.  For $\eta<\gamma+\gamma'$, choose an extremely nice $(\theta'_\eta, e'_\eta):\aaa_{\eta+1}\to \aaa_{\zeta_1(\phi(\eta))}$ and define $\theta$ by letting $\theta|_{\aaa_{\eta+1}}= \theta_{\zeta_1(\phi(\eta))}\circ \theta'_\eta$ and $e|_{MAX(\aaa_{\eta+1})}= e_{\zeta_1(\phi(\eta))} \circ e'_\eta$.  

\end{proof}

\begin{proof}[Proof of Theorem \ref{Minkowski}] $(i)$ It is clear that $Sz(K), Sz(L)\leqslant Sz(K\cup L)$.  Assume $Sz(K\cup L)>\xi$.  By Lemma \ref{lemma2}, choose $(x_\alpha)_{\alpha\in \aaa_\xi}$ and $\ee>0$ so that if $\alpha=(t, (U_1, \ldots, U_n))\in \aaa_\xi$, $x_\alpha\in U_n$ and so that for each $\alpha\in \aaa_\xi$, $(x_{\alpha|_i})_{i=1}^{\alpha}\in \hhh_\ee^{K\cup L}$.  For each $\alpha\in MAX(\aaa_\xi)$, choose $x^*_\alpha\in K\cup L$ so that for each $1\leqslant i\leqslant |\alpha|$, $x^*_\alpha(x_{\alpha|_i})\geqslant \ee$.  Let $\ccc^1$ consist of those $\alpha\in MAX(\aaa_\xi)$ so that $x^*_\alpha\in K$ and $\ccc^2$ consist of those $\alpha\in MAX(\aaa_\xi)$ so that $x^*_\alpha\in L$.  By Lemma \ref{useful lemma}, there exists $j\in \{1, 2\}$ and an extremely nice $(\theta, e):\aaa_\xi\to \aaa_\xi $ so that $e(MAX(\aaa_\xi))\subset \ccc^j$.  If $j=1$, for each $\alpha\in MAX(\aaa_\xi)$, $x^*_{e(\alpha)}\in K$ and $x^*_{e(\alpha)}(x_{\theta(\alpha|_i)})\geqslant \ee$ for each $1\leqslant i\leqslant |\alpha|$.  Since $\theta$ is nice, we deduce that $(x_{\theta(\alpha)})_{\alpha\in \aaa_\xi}$ and $(x^*_{e(\alpha)})_{\alpha\in MAX(\aaa_\xi)}$ witness the fact that $o_w(\hhh^K_\ee)>\xi$.  If $j=2$, we similarly deduce that $o_w(\hhh^L_\ee)>\xi$. Since $\xi$ was arbitrary, this completes $(i)$.

$(ii)$ This is an immediate consequence of Lemma \ref{lemma coup de grace}$(iii)$ and Theorem \ref{main theorem}.

$(iii)$ We will show that $o_w(\hhh^{K+L}_\ee)\leqslant o_w(\hhh^K_{\ee/2})\oplus o_w(\hhh^L_{\ee/2})$ for all $\ee>0$.  Assume $o_w(\hhh^K_{\ee/2})=\eta_1\in \ord$ and $o_w(\hhh^L_{\ee/2})= \eta_2\in \ord$.  Suppose $\xi=\eta_1\oplus \eta_2<o_w(\hhh^{K+L}_\ee)$.  Choose $(x_\alpha)_{\alpha\in \aaa_\xi}$ according to Lemma \ref{lemma2}. For each $\alpha\in MAX(\aaa_\xi)$, choose $f^1_\alpha\in K$ and $f^2_\alpha\in L$ so that for each $1\leqslant i\leqslant |\alpha|$, $(f^1_\alpha+f^2_\alpha)(x_{\alpha|_i})\geqslant \ee$.   By Lemma \ref{useful lemma}, there exist $\zeta_1, \zeta_2$ with $\zeta_1\oplus\zeta_2=\xi$ and for $j=1,2$ some extremely nice $(\theta_j, e_j):\aaa_{\zeta_j}\to \aaa_\xi$.  Then $(x_{\theta_1(\alpha)})_{\alpha\in \aaa_{\zeta_1}}$ and $(f^1_{e_1(\alpha)})_{\alpha\in MAX(\aaa_{\zeta_1})}$ can be used to deduce that $\zeta_1< o_w(\hhh^K_{\ee/2})$.  Similarly, $(x_{\theta_2(\alpha)})_{\alpha\in \aaa_{\zeta_2}}$ and $(f^2_{e_2(\alpha)})_{\alpha\in \aaa_{\zeta_2}}$ used to deduce that $\zeta_2<o_w(\hhh^L_{\ee/2})$.   Then $\zeta_1<\eta_1$ and $\zeta_2<\eta_2$, whence $$\xi=\zeta_1\oplus \zeta_2 < \eta_1\oplus \eta_2=\xi,$$ a contradiction.

$(iv)$ If both sets are norm compact, then so is the sum, and the result follows from Proposition \ref{special form}.  If the maximum is $\infty$, the result follows from $(ii)$.  Otherwise $\max\{Sz(K), Sz(L)\}=\omega^\xi$ for some $\xi>0$, and $o_w(\hhh^K_{\ee/2}), o_w(\hhh^L_{\ee/2})$, and therefore $o_w(\hhh^K_{\ee/2})\oplus o_w(\hhh^L_{\ee/2})$, are less than $\omega^\xi$ by Proposition \ref{special form}.  In this case, the result follows from $(iii)$.

\end{proof}

\begin{remark} We wish to thank P.A.H. Brooker for bringing the following observation to our attention.  Suppose $\phi:\ord\times \ord\to \ord$ is such that $Sz(K+L)\leqslant \phi(Sz(K), Sz(L))$ for arbitrary Banach spaces $X$ and arbitrary, $w^*$ compact, non-empty subsets $K, L\subset X^*$.  Then if $K\subset X^*$ and $L\subset Y^*$ are non-empty and $w^*$ compact, $K\times L= K+L\subset X^*\oplus Y^*$, and we deduce $Sz(K\times L)\leqslant \phi(Sz(K), Sz(L))$.  This is because $K\subset X^*$ has the same Szlenk index as $K+\{0\}\subset X^*\oplus Y^*$, and similarly for $L\subset Y^*$.  Conversely, if $\psi:\ord\times \ord\to \ord$ is such that $Sz(K\times L)\leqslant \psi(Sz(K), Sz(L))$ for arbitrary Banach spaces $X, Y$ and arbitrary $w^*$ compact, non-empty subsets $K\subset X^*$ and $L\subset Y^*$.  Then for any Banach space $X$, let $D:X\to X\oplus X$ be defined by $Dx=(x,x)$, so $D^*(x^*, y^*)=x^*+y^*$.  Then $D^*(K\times L)=K+L\subset X^*$.  It is easy to see (and we will offer a rigorous proof later) that $Sz(K+L)\leqslant Sz(D^*(K\times L))\leqslant Sz(K\times L)$, we deduce that $Sz(K+L)\leqslant \psi(Sz(K), Sz(L))$.  Thus any estimates for the Szlenk index of a Minkowski sum in terms of the indices of the individual summands yield the same estimates of Cartesian products in terms of the individual factors, and conversely.  

\end{remark}

\subsection{Szlenk index of an operator} If $X, Y$ are Banach spaces and $A:X\to Y$ is an operator, we define the Szlenk index of the operator $A$ to be $Sz(A)=Sz(A^*B_{Y^*})$.  The main result of this subsection is the following, the first statement of which was originally shown by Brooker \cite{B}, where the slicing definition of the Szlenk index was used.    The second statement of the Theorem was shown by Hajek and Lancien \cite{HL} using the slicing definition, as well as by Odell, Schlumprecht, and Zs\'{a}k \cite{OSZ} in the case that $X$ and $Y$ are separable.

\begin{theorem} For every $\xi\in \emph{\ord}$, the class of operators having Szlenk index not exceeding $\omega^\xi$ is a closed operator ideal.  In particular, $Sz(X\oplus Y)= \max\{Sz(X), Sz(Y)\}$ for any $X, Y\in \emph{\Ban}$.  

\label{operator ideal}
\end{theorem}

By Proposition \ref{special form}, the Szlenk index of an operator must be of the form $\omega^\xi$ for some $\xi\in \ord$, so considering the classes bounded only by gamma numbers $\omega^\xi$ loses no generality.  

Before proceeding to the proof, we separate the following result which was promised above.  

\begin{lemma} Let $X,Y,Z$ be Banach spaces and let $A:X\to Y$, $B:Y\to Z$ be operators.  \begin{enumerate}[(i)]\item For any $K\subset Y^*$ $w^*$ compact and non-empty, $Sz(A^*K)\leqslant Sz(K)$. \item $Sz(AB)\leqslant \min \{Sz(A),Sz(B)\}$.   \end{enumerate}   

\label{compose}
\end{lemma}

\begin{proof}$(i)$ Since it is clear that for any $c, \ee>0$ and any $L\subset X^*$ $w^*$ compact and non-empty, $\hhh^L_\ee= \hhh^{cL}_{c\ee}$, we may assume that $\|A\|\leqslant 1$.  Given $\hhh\subset B_X^{<\nn}$, let $A(\hhh)=\{S(t): t\in \hhh\}$, where $A((x_i)_{i=1}^n)= (Ax_i)_{i=1}^n$ and $A(\varnothing)=\varnothing$.  We claim that $A((\hhh^{A^*K}_\ee)^\xi_w)\subset (\hhh^K_\ee)_w^\xi$ for any $\xi\in \ord$, which will finish $(i)$.  We prove the result by induction on $\xi$, noting that the base case and limit ordinal cases are trivial.  Fix $s\in (\hhh^{A^*K}_\ee)_w^{\xi+1}$.  Choose a weakly null net $(x_\lambda)\subset B_X$ so that $s\cat x_\lambda\in (\hhh^{A^*K}_\ee)_w^\xi$ for all $\lambda$.  Then $(Ax_\lambda)_\lambda\subset B_Y$ is weakly null, and by the inductive hypothesis, $A(s)\cat A(x_\lambda)\in (\hhh^K_\ee)_w^\xi$, yielding that $A(s)\in (\hhh^K_\ee)_w^{\xi+1}$.

$(ii)$  By $(i)$, $$Sz(AB)=Sz(B^*A^*B_{Z^*})\leqslant Sz(A^*B_{Z^*})=Sz(A).$$ As in $(i)$, we may assume $\|A\|\leqslant 1$, so $B^*A^*B_{Z^*} \subset B^*B_{Y^*}$, yielding $Sz(AB)\leqslant B$.

\end{proof}

\begin{proof}[Proof of Theorem \ref{operator ideal}]

For $X,Y\in \Ban$, let $\mathfrak{Sz}_\xi(X,Y)$ denote the operators from $X$ to $Y$ having Szlenk index not exceeding $\omega^\xi$.  Let $\mathfrak{Sz}_\xi$ consist of the class of all operators lying in one of the components $\mathfrak{Sz}_\xi(X,Y)$, $X,Y\in\Ban$.

First, we note that for $(x_i)_{i=1}^n\in B_X^{<\nn}$, $(x_i)_{i=1}^n\in \hhh^{A^*B_{Y^*}}_\ee$ if and only if there exists $y^*\in B_{Y^*}$ so that $A^*y^*(x_i)=y^*(Ax_i)\geqslant \ee$ for each $1\leqslant i\leqslant n$.  By the Hahn-Banach theorem, this is equivalent to the condition that every convex combination of $(Ax_i)_{i=1}^n$ has norm at least $\ee$.  We will use this characterization throughout the proof.  

We know from Proposition \ref{special form} and Schauder's theorem that the members of $\mathfrak{Sz}_0(X,Y)$ are precisely the compact operators from $X$ to $Y$.  Thus $\mathfrak{Sz}_\xi$ contains all finite rank operators for any $\xi\in \ord$.  

If $A, B:X\to Y$ both have Szlenk index not exceeding $\omega^\xi$, note that $(A+B)^*B_{Y^*}\subset A^*B_{Y^*}+ B^* B_{Y^*}$, and $A^* B_{Y^*}, B^* B_{Y^*}$ are convex.  By \ref{Minkowski}$(iv)$, $$Sz((A^*+B^*)B_{Y^*})\leqslant Sz(A^*B_{Y^*}+B^*B_{Y^*})=\max\{Sz(A), Sz(B)\}\leqslant \omega^\xi.$$  Thus $\mathfrak{Sz}_\xi(X,Y)$ is closed under finite sums.

By Lemma \ref{compose}, for any $A:W\to X$, $B:X\to Y$, and $C:Y\to Z$, $Sz(ABC)\leqslant Sz(B)$, so that if $B\in \mathfrak{Sz}_\xi$, $ABC\in \mathfrak{Sz}_\xi$.  

Last, assume $A:X\to Y$ is an operator with $Sz(A)>\omega^\xi$.   Then there exists $\ee>0$ so that $o_w(\hhh^{A^*B_{Y^*}}_{2\ee})>\xi$.  Then for any $B:X\to Y$ with $\|A-B\|< \ee$, $\hhh^{B^*B_{Y^*}}_\ee \supset \hhh^{A^*B_{Y^*}}_{2\ee}$.  This is because if $(x_i)_{i=1}^n\in B_X^{<\nn}$ is such that all convex combinations of $(Ax_i)_{i=1}^n$ have norm at least $2\ee$, then since $(Bx_i)_{i=1}^n$ is an $\ee$-perturbation of $(Ax_i)_{i=1}^n$, all convex combinations of $(Bx_i)_{i=1}^n$ have norm at least $\ee$.  Of course, this means that for all $\zeta\in \ord$, $(\hhh^{B^*B_{Y^*}}_\ee)_w^\zeta\supset (\hhh^{A^*B_{Y^*}}_{2\ee})_w^\zeta$, and $\omega^\xi< o_w(\hhh^{A^*B_{Y^*}}_{2\ee}) \leqslant o_w(\hhh^{B^*B_{Y^*}}_\ee)$.  Thus we have shown that the complement $\mathfrak{B}(X,Y)\setminus \mathfrak{Sz}_\xi(X,Y)$ of $\mathfrak{Sz}_\xi(X,Y)$ in the space of operators $\mathfrak{B}(X,Y)$ from $X$ to $Y$ is norm open, whence $\mathfrak{Sz}_\xi(X,Y)$ is norm closed.  

The second statement follows from the fact that for any $X, Y\in \Ban$, if $P_X, P_Y$ are the projections from $X\oplus Y$ to $X, Y$, respectively, $Sz(P_X), Sz(P_Y)\leqslant \max\{Sz(X), Sz(Y)\}$ and $$Sz(X\oplus Y)= Sz(P_X+P_Y)\leqslant \max\{Sz(P_X), Sz(P_Y)\}\leqslant \max\{Sz(X), Sz(Y)\}.$$  Since $I_X, I_Y$ both factor through $I_{X\oplus Y}$, the ideal property gives that $Sz(X), Sz(Y)\leqslant Sz(X\oplus Y)$.  

\end{proof}

\subsection{Combinatorial interpretation of sums}

In this subsection, we discuss the results above in terms of finite colorings, generalizing the specific applications above.  We omit the proofs, since they are inessential modifications of the results above.  

\begin{proposition} Let $D$ be a directed set, $\xi\in \emph{\ord}$, $n\in \nn$.  \begin{enumerate}[(i)]\item Suppose that for $1\leqslant j\leqslant n$, $\ccc^j\subset MAX(\ttt^D_\xi)$, and $\cup_{j=1}^n \ccc^j=MAX(\ttt_\xi^D)$.  Then there exists an extremely nice $(\theta, e):\ttt_\xi^D\to \ttt_\xi^D$ and $j\in \{1, \ldots, n\}$ so that $e(MAX(\ttt_\xi^D))\subset \ccc^j$. 

\item If for each $1\leqslant j\leqslant n$, $\ccc^j\subset \ttt^D_\xi$ is downward closed with respect to $\preceq$, and if $\cup_{j=1}^n \ccc^j=\ttt_\xi^D$, then there exists $1\leqslant j\leqslant n$ and a nice $\theta:\ttt^D_\xi\to \ttt^D_\xi$ so that $\theta(\ttt^D_\xi)\subset \ccc^j$. 

\item Suppose that for each $\tau\in \ttt^D_\xi$, $\ccc^j(\tau)$, $1\leqslant j\leqslant n$ are such that $\cup_{n=1}^n \ccc^j(\tau)=\{\tau_0\in MAX(\ttt^D_\xi): \tau\preceq \tau_0\}$.  Then there exist $\zeta_1, \ldots, \zeta_n$ so that $\zeta_1\oplus \ldots \oplus \zeta_n=\xi$ and for each $1\leqslant j\leqslant n$ extremely nice $(\theta_j, e_j):\ttt^D_{\zeta_j}\to \ttt^D_\xi$ so that for each $1\leqslant j\leqslant n$, each $\tau\in MAX(\ttt^D_{\zeta_j})$, $e_j(\tau)\in \cap_{i=1}^{|\tau|} \ccc^j(\theta(\tau|_i))$. 

\item Suppose that for each $1\leqslant j\leqslant n$, $\ccc^j\subset \ttt^D_\xi$, and $\cup_{j=1}^n \ccc^j=\ttt_\xi^D$.  Then there exist $\zeta_1, \ldots, \zeta_n$ so that $\zeta_1\oplus \ldots \oplus\zeta_n= \xi$ and nice $\theta_j:\ttt^D_{\zeta_j}\to \ttt^D_\xi$ so that for each $1\leqslant j\leqslant n$, $\theta_j(\ttt^D_{\zeta_j})\subset \ccc^j$. \end{enumerate}

\label{sum summary}

\end{proposition}

We first note that the result for any number of colors follow by iterating the result for two colors.  We note that $(ii)$ is an easy consequence of $(i)$ and $(iv)$ is an easy consequence of $(iii)$.  We proved $(i)$ in the case that $\ttt^D_\xi=\aaa_\xi$ in Lemma \ref{useful lemma}$(ii)$.  The general case is essentially the same.  Similarly, Lemma \ref{useful lemma}$(iii)$ is a special case of $(iii)$ of Proposition \ref{sum summary}.

\section{Product estimate applications}

\subsection{Relation to the Bourgain $\ell_1$-index}

In \cite{Bo}, Bourgain defined the Bourgain $\ell_1$ index of a Banach space.  This index measures the local complexity of $\ell_1$ structure within a given Banach space in terms of the orders of trees the branches of which are equivalent to the $\ell_1$ basis with a uniform constant of equivalence.  A given Banach space contains an isomorphic copy of $\ell_1$ if if and only if one of these trees is ill-founded.  The following definition of the Bourgain $\ell_1$-index of an operator was defined in \cite{BCF}.  For an operator $A:X\to Y$ and $\ee>0$, we let $T(A,X,Y,\ee)$ consist of all $(x_i)_{i=1}^n\in B_X^{<\nn}$ so that for all scalars $(a_i)_{i=1}^n$, $\|\sum_{i=1}^n a_iAx_i\|\geqslant \ee\sum_{i=1}^n |a_i|$.  By convention, we include the empty sequence in $T(A,X,Y,\ee)$.  We let $I(A)=\sup_{\ee>0}o(T(A,X,Y,\ee))$.  This index measures the complexity of local $\ell_1$ structures in $X$ which are preserved by $A$.  Then $I(A)<\infty$ if and only if $A$ does not preserve an isomorph of $\ell_1$. This index generalizes the $\ell_1$ index of a Banach space, since the $\ell_1$ index of a Banach space coincides with the index of the identity operator of that Banach space.  We remark here that if $(x_i)_{i=1}^n\in T(A,X,Y, \ee)^\xi$ and $y_j=\sum_{i=p_{j-1}+1}^{p_j} a_ix_i$ for some $0=p_0<\ldots <p_m=n$ and scalars $(a_i)_{i=1}^n$ so that $\sum_{i=p_{j-1}+1}^{p_j}|a_i|=1$ for each $1\leqslant j\leqslant m$, then $(y_j)_{j=1}^m\in T(A,X,Y,\ee)^\xi$.  This can be easily shown by induction on $\xi$.  

The main result of this subsection is the following.  We draw the reader's attention to \cite{AJO}, where a similar result was shown for the Szlenk and Bourgain $\ell_1$ indices of a Banach space, not an operator, assuming the space is separable and has a sequentially ordered basis.  

\begin{theorem} Suppose $A:X\to Y$ is an operator and $Y$ has an unconditional basis $(e_i)_{i\in I}$.  Then $Sz(A)\leqslant I(A)\leqslant \omega Sz(A).$  In particular, if $Sz(A)\geqslant \omega^\omega$, $Sz(A)=I(A)$.  

\label{Bourgain}
\end{theorem}

\begin{proof} By renorming $Y$, we may assume the basis $(e_i)_{i\in I}$ is $1$-unconditional, noting that this does not change $Sz(A)$ or $I(A)$.  This is because by Theorem \ref{operator ideal} the Szlenk index is unchanged by composing $A$ with an isomorphism on $Y$, and the same is true of $I(A)$ by results from \cite{BCF}.  

We first prove that $I(A)\leqslant \omega Sz(A)$.  To do this, we will prove that $$T(A, X, Y, \ee)^{\omega\xi}\subset (\hhh^{A^*B_{Y^*}}_\ee)^\xi_w$$ for each $\xi\in \ord$.  As we remarked in the previous subsection, a non-empty sequence $(x_i)_{i=1}^n$ lies in $\hhh^{A^*B_{Y^*}}$ if all convex combinations of $(Ax_i)_{i=1}^n$ have norm at least $\ee$.  This easily implies that $T(A,X,Y,\ee)\subset \hhh^{A^*B_{Y^*}}_\ee$, which is the $\xi=0$ case.  The limit ordinal case is trivial.  Assume $T(A,X,Y, \ee)^{\omega\xi}\subset (\hhh^{A^*B_{Y^*}}_\ee)^\xi_w$.  If $T(A,X,Y, \ee)^{\omega(\xi+1)}=\varnothing\subset (\hhh^{A^*B_{Y^*}}_\ee)^{\xi+1}_w$, we are done.  So assume $t\in T(A,X,Y,\ee)^{\omega(\xi+1)}=(T(A,X,Y,\ee)^{\omega\xi})^\omega$.  This simply means that for any $n\in \nn$, there exists $(x_i)_{i=1}^n\in B_X^{<\nn}$ so that $t\cat (x_i)_{i=1}^n\in T(A,X,Y,\ee)^{\omega\xi}$.  Fix $U\in \mmm$ and write $U=\{x: |x^*(x)|<\delta \text{\ }\forall x^*\in F\}$, where $F$ is finite.  Fix $n>|F|$, and $(x_i)_{i=1}^n$ so that $t\cat (x_i)_{i=1}^n\in T(A,X,Y,\ee)^{\omega\xi}$.  By a dimension argument, we may choose $x=\sum_{i=1}^n a_ix_i$ where $\sum_{i=1}^n |a_i|=1$ and so that $x^*(x)=0$ for each $x^*\in F$.  Thus $x\in U\cap B_X$.  By our remark in the paragraph preceding the statement of the theorem, since $t\cat (x_i)_{i=1}^n\in T(A,X,Y, \ee)^{\omega\xi}$, $t\cat x\in T(A,X,Y,\ee)^{\omega\xi}$ and, by the inductive hypothesis, $t\cat x\in (\hhh^{A^*B_{Y^*}}_\ee)_w^\xi$.  Since $U$ was arbitrary, this guarantees that $t\in (\hhh^{A^*B_{Y^*}}_\ee)_w^{\xi+1}$.  This completes the claim and shows that $I(A)\leqslant \omega Sz(A)$.

Next, assume $o_w(\hhh^{A^*B_{Y^*}}_\ee)>\xi$ for some $\xi\in \ord$ and $\ee>0$.  By Lemma \ref{lemma2}, we may choose $(x_\alpha)_{\alpha\in \aaa_\xi}$ so that for $(t, (U_1, \ldots, U_n))\in \aaa_\xi$, $x_\alpha\in U_n$, and for each $\alpha\in \aaa_\xi$, $(x_{\alpha|_i})_{i=1}^{|\alpha|}\in \hhh^{A^*B_{Y^*}}_\ee$.   For $J\subset I$, let $P_J:Y\to Y$ be the projection $P_J\sum_{i\in I}a_ie_i=\sum_{i\in J}a_ie_i$.  Define a monotone $\theta:\ttt_\xi\to \aaa_\xi$ so that for each $t\in \ttt_\xi$, there exists $\sigma\in \mmm^{<\nn}$ so that $\theta(t)=(t, \sigma)$ and a finite set $I_t\subset I$ so that \begin{enumerate}[(i)]\item $\|Ax_{\theta(t)} - P_{I_t}Ax_{\theta(t)}\|<\ee/5$, \item for each $t\in \ttt_\xi$, $\|P_{\cup_{i=1}^{|t|-1} I_{t|_i}}Ax_{\theta(t)}\|<\ee/5$, \item $I_s\subset I_t$ for each $s\in \ttt_\xi$ with $s\prec t$.   \end{enumerate} More precisely, for $t\in \ttt_\xi$ with $|t|=1$, let $\theta(t)=(t, U)$ for some $U\in \mmm$ and choose $I_t\subset I$ finite so that $\|Ax_{\theta(t)} - P_{I_t}Ax_{\theta(t)}\|<\ee/5$.  Next, if $\theta(s)$ and $I_s$ have been defined for each $s\in \ttt_\xi$ with $|s|<n$, and if $t\in \ttt_\xi$ with $|t|=n$, let $\sigma\in \mmm^{<\nn}$ be such that $\theta(p(t))=(p(t), \sigma)$.   Since $(x_{(t, \sigma\cat U)})_{U\in \mmm}$ is a weakly null net, and since $P_{\cup_{i=1}^{n-1} I_{t|_i}}A$ is compact, there exists $U\in \mmm$ so that $\|P_{\cup_{i=1}^{n-1}I_{t|_i}}Ax_{(t, \sigma\cat U)}\|<\ee/5$.  Define $\theta(t)=(t, \sigma\cat U)$.  Choose $I_t$ so that $\|Ax_{\theta(t)}- P_{I_t}Ax_{\theta(t)}\|<\ee/5$.  This completes the recursive definition of $\theta(t)$ and $I_t$.  

For $t\in \ttt_\xi$, let $y_t=P_{I_t\setminus \cup_{i=1}^{|t|-1}I_{t|_i}} Ax_{\theta(t)}$.  Note that for each $t\in \ttt_\xi$, $(y_{t|_i})_{i=1}^{|t|}$ is disjointly supported in $Y$ and $\|y_t-Ax_{\theta(t)}\|\leqslant \|P_{I\setminus I_t} x_{\theta(t)}\| + \|P_{\cup_{j=1}^{|t|-1} I_{t|_j}} x_{\theta(t)}\|<2\ee/5$.  Fix $t\in \ttt_\xi$ and positive scalars $(a_i)_{i=1}^{|t|}$.  Then $$\|\sum_{i=1}^{|t|} a_iy_{t|_i}\| \geqslant \|\sum_{i=1}^{|t|} a_iAx_{\theta(t|_i)}\| - \sum_{i=1}^{|t|} a_i2\ee/5 \geqslant 3\ee/5\sum_{i=1}^{|t|}a_i.$$  Here we have used the fact that $(x_{\theta(t|_i)})_{i=1}^{|t|}\in \hhh^{A^*B_{Y^*}}_\ee$, so that any convex combination of $(Ax_{\theta(t|_i)})_{i=1}^{|t|}$ has norm at least $\ee$, and by homogeneity $\|\sum_{i=1}^{|t|} a_i Ax_{\theta(t|_i)}\|\geqslant \ee \sum_{i=1}^{|t|}a_i$ for any positive scalars $(a_i)_{i=1}^{|t|}$.   But since $(y_{t|_i})_{i=1}^{|t|}$ is a disjointly supported sequence in a $1$-unconditional basis, $$\|\sum_{i=1}^{|t|}a_i y_{t|_i}\|\geqslant 3\ee/5\sum_{i=1}^{|t|}|a_i|$$ for any scalars $(a_i)_{i=1}^{|t|}$.  But then $$\|\sum_{i=1}^{|t|} a_iAx_{\theta(t|_i)} \|\geqslant \|\sum_{i=1}^{|t|} a_iy_{t|_i}\|-2\ee/5\sum_{i=1}^{|t|}|a_i| \geqslant \ee/5\sum_{i=1}^{|t|}|a_i|$$ for any scalars $(a_i)_{i=1}^{|t|}$.  Then by \cite{C}, $(x_{\theta(t)})_{t\in \ttt_\xi}$ witnesses the fact that $I(A)>\xi$, which shows $Sz(A)\leqslant I(A)$.  

We turn now to the second statement.  We have shown that if $I(A)=\infty$ if and only if $Sz(A)=\infty$. If $Sz(A)\geqslant \omega^\omega$, then $A$ cannot be compact.     Therefore we must only deal with the case that $I(A), Sz(A)<\infty$ and $A$ is not compact.  But it is known in this case \cite{BCF}   that there exists $\eta\in \ord$ so that $I(A)=\omega^\eta$. By Proposition \ref{special form}, there exists $\xi\in \ord$ so that $Sz(A)=\omega^\xi$.   The inequalities above guarantee that $\xi\leqslant \eta\leqslant 1+\xi$ and, if $\xi\geqslant \omega$, $1+\xi=\xi=\eta$.

\end{proof}

\subsection{Infinite direct sums} Suppose that $(e_i)_{i\in I}$ is a $1$-unconditional  basis for $E$.  Assume also that for each $i\in I$, $X_i$ is a Banach space, and let $X=(\oplus_{i\in I}X_i)_E$. Let $\pi:X\to E$ be the function taking $(x_i)_{i\in I}$ to $\sum_{i\in I}\|x_i\|e_i$.  Let $\pi_*:X^*\to E^*$ be defined by $\pi_*(x_i^*)_{i\in I}= \sum_{i\in I}\|x_i^*\|e_i^*$. Recall that $\pi_*$ is well-defined, although $\pi_*(x_i^*)_{i\in I}$ is only guaranteed to be a formal series, and not necessarily countably non-zero or norm convergent. We note that $\|x\|=\|\pi(x)\|$ for all $x\in X$ and $\|x^*\|= \|\pi_*(x^*)\|$ for all $x^*\in X^*$.   For $J\subset I$, we let $P_J$ denote both the projection $P_J:E\to E$ defined by $P_J\sum_{i\in I}a_ie_i=\sum_{i\in J}a_ie_i$, as well as the projection $P_J:X\to X$ defined by $P_J(x_i)_{i\in I}=(1_J(i)x_i)_{i\in I}$.   For each $i\in I$, let $L_i$ be a symmetric, non-empty, convex, $w^*$ compact subset of $X_i^*$.  Assume also that $L\subset [e_i^*:i\in I]\subset E^*$ is $w^*$ compact, unconditional, convex, and non-empty.  By unconditional, we mean that $\sum_{i\in I}a_i e_i^*\in L$ if and only if $\sum_{i\in I}\ee_i a_i e_i^*\in L$ for all $(\ee_i)_{i\in I}\in \{\pm 1\}^I$. We let $K=\{x^*\in X^*\cap \prod_{i\in I}L_i: \pi_* (x^*)\in L\}$.  It is easy to see that this set is $w^*$ compact, convex, symmetric, and non-empty.  

The main result of this subsection is the following.  We draw the reader's attention to \cite{B2}, where a similar result was shown in the case that $E=F=\ell_p(I)$ for $1\leqslant p\leqslant \infty$ or $E=F=c_0(I)$, where $(e_i)_{i\in I}=(f_i)_{i\in I}=(1_{\{i\}})_{i\in I}$.

\begin{theorem} With $X$, $L$, $L_i$, and $K$ as above, there exists a constant $C>1$ so that $$Sz_\ee(K)\leqslant (\sup_{i\in I} Sz(L_i)) Sz_{\ee/C}(L).$$  Consequently, $Sz(K)\leqslant (\sup_{i\in I}Sz(L_i))Sz(L)$.  

\label{infinite direct sum}

\end{theorem}

\begin{proof} The result follows from Theorem \ref{operator ideal} if $|I|<\infty$, so assume $I$ is infinite.  Recall that $[I]^{<\nn}$ denotes the finite subsets of $I$, and let this set be directed by inclusion. Recall also that $|x|_K=\sup_{x^*\in K}|x^*(x)|$, and that for $(x_j)_{j=1}^n\in B_X^{<\nn}$, $(x_j)_{j=1}^n\in \hhh^K_\ee$ if and only if $|x|_K\geqslant \ee$ for all convex combinations $x$ of $(x_j)_{j=1}^n$.  

Note that since $L\subset [e_i^*:i\in I]$, and since $K$ is also unconditional and convex, the set of tuples $(x^*_i)_{i\in I}\in K$ so that $x^*_i=0$ for all but finitely many $i\in I$ is norm dense in $K$.   For this reason, if $(x_J)_{J\in [I]^{<\nn}}\subset B_X$ is such that $|P_J x_J|_K< 2^{-|J|}$ for all $J\in [I]^{<\nn}$, then for any $f\in K$, the net $(f(x_J))_{J\in [I]^{<\nn}}$ converges to zero. 

Choose $R>0$ so that $L\subset R B_{E^*}$, and note that $K\subset R B_{X^*}$.  Let $\xi=\sup_{i\in I}Sz(B_i)$. If $\xi=\infty$, there is nothing to prove, so assume $\xi\in \ord$.  For $\zeta\in \ord$, let $\Gamma_\zeta=\ttt^{[I]^{<\nn}}_\zeta$.   We prove by induction on $\zeta\in \ord$ that if $s\in (\hhh^K_\ee)_w^{\xi\zeta}$, then there exists $(x_\gamma)_{\gamma \in \Gamma_\zeta}$ so that for all $\gamma\in \Gamma_\zeta$, $(x_{\gamma|_i})_{i=1}^{|\gamma|}\in \hhh^K_\ee(s)$ and, if $\gamma=(t, (J_1, \ldots, J_n))\in \Gamma_\zeta$, $|P_{J_n} x_\gamma|_K< 2^{-|J_n|}$.  In particular, for each $t\in \ttt_\zeta$ and $\sigma\in ([I]^{<\nn})^{<\nn}$ with $|t|=|\sigma|+1$, $(x_{(t, \sigma\cat J)})_{J\in [I]^{<\nn}}\subset B_X$ is a net which is pointwise null on $K$ by our remark above.  The only non-trivial case of the induction is the successor case.  Assume the result holds for some $\zeta$ and assume $s\in (\hhh^K_\ee)^{\xi(\zeta+1)}_w$.  Since $\hhh:=(\hhh^K_\ee(s))_w^{\xi\zeta}$ is such that $o_w(\hhh)>\xi$, by Lemma \ref{lemma2} we may select an $\mmm$ tree $(u_\alpha)_{\alpha\in \aaa_\xi}$ in $\hhh$.  Fix $J\in [I]^{<\nn}$.  Note that $P_J^* K$ is contained in the Minkowski sum $R(L_{j_1}+\ldots +L_{j_k})$, where $J=\{j_1, \ldots, j_k\}$. By Theorem \ref{Minkowski}, $Sz(P_J^* K)\leqslant \xi$.  If for every $\alpha\in\aaa_\xi$, every convex combination $x$ of $(u_{\alpha|_i})_{i=1}^{|\alpha|}$ satisfies $|P_J x|_K \geqslant 2^{-|J|}$, then $(P_J x_\alpha)_{\alpha\in \aaa_\xi}\subset B_{\oplus_{i\in J} X_i}$ would give that $Sz(P_J^* K)>\xi$, a contradiction. Here we have used that $|P_Jx|_K = |P_Jx|_{P^*_JK}$.  Therefore there must exist some $\alpha\in \aaa_\xi$ and some convex combination $x^J$ of $(u_{\alpha|_i})_{i=1}^{|\alpha|}$ so that $|P_J x^J|_K< 2^{-|J|}$.  Since $x^J$ is a convex combination of a member of $\hhh$, the length $1$ sequence $(x^J)$ is a member of $\hhh$.  This means that $s\cat x^J \in (\hhh_\ee^K)_w^{\xi\zeta}$.  By the inductive hypothesis, there exists $(x^J_\gamma)_{\gamma\in \Gamma_\zeta}$ satisfying the conclusions with $s$ replaced by $s\cat x^J$.  We then define $(x_\gamma)_{\gamma\in \Gamma_{\zeta+1}}$ by $x_{(\zeta+1, J)}=x^J$ and, if $\zeta>0$, $x_{(\zeta+1,J)\cat (t,\sigma)}= x^J_{(t, \sigma)}$.  This completes the induction.

Next, fix $0<\delta<\ee_0<\ee$ and $0<\mu< (\ee-\ee_0)/2$.  Let $\zeta=o_{(L, \delta)}(\hhh^L_{\ee_0})$ and assume $\zeta\in \ord$, otherwise the result is trivial.  To obtain a contradiction, assume $o_w(\hhh^K_\ee)>\xi\zeta$.  By the induction above, there exists $(x_\gamma)_{\gamma\in \Gamma_\zeta}$ so that for each $\gamma=(t, \sigma\cat J)\in \Gamma_\zeta$, $|P_J x_\gamma|_K< 2^{-|J|}$.  Define $m:\Gamma_\zeta\to [I]^{<\nn}$ and a nice $\theta:\Gamma_\zeta\to \Gamma_\zeta$ by induction on $|\gamma|$ as follows: If $|\gamma|=1$, write $\gamma=(t, J_0)$ and let $\theta(\gamma)=(t, J)$, where $J_0\subset J$ and $|J|> \log_2(\mu^{-1})$.  Choose $m(\gamma)\in [I]^{<\nn}$ so that $\|P_{I\setminus m(\gamma)}x_\gamma\|< \mu/R$.  Since $K\subset RB_{X^*}$, $|P_{I\setminus m(\gamma)} x_\gamma|_K< \mu$.  

Next, assume $m(\gamma)$ and $\theta(\gamma)$ have been defined for each $\gamma\in \Gamma_\zeta$ with $|\gamma|<n$ so that if $\gamma=(t, \sigma)$, $\theta(\gamma)=(t, \sigma_0)$ for some $\sigma_0\in ([I]^{<\nn})^{<\nn}$.   Fix $\gamma\in \Gamma_\zeta$ with $|\gamma|=n$ (if such a $\gamma$ exists, otherwise we are already done with the definitions of $m$ and $\theta$).  Write $\gamma=(t, \sigma \cat J_0)$ and $\theta(p(\gamma))=(p(t), \sigma_0)$.  Choose $J\in [I]^{<\nn}$ so that $J_0 \subset J$, $\cup_{j=1}^{n-1} m_{\gamma|_j} \subset J$, and $|J|>\log_2(\mu^{-1})$.  Let $\theta(\gamma)=(t, \sigma_0\cat J)$.  Choose $m(\gamma)\in [I]^{<\nn}$ so that $m(\gamma|_j)\subset m(\gamma)$ for each $1\leqslant j<|\gamma|$ and so that $\|P_{I\setminus m(\gamma)} x_{\theta(\gamma)}\|<\mu/R$.  Note that $|P_{I\setminus m(\gamma)} x_{\theta(\gamma)}|_K< \mu$. This completes the recursive construction of $m$ and $\theta$.  Note that for any $\gamma\in \Gamma_\zeta$, if $\theta(\gamma)=(t, (J_1, \ldots, J_n))$, then $J_n\supset \cup_{j=1}^{n-1} m(\gamma|_j)$.   Then $|P_{\cup_{j=1}^{n-1}m(\gamma|_j)} x_{\theta(\gamma)}|_K \leqslant |P_{J_n} x_{\theta(\gamma)}|_K < 2^{-|J_n|}< \mu$.  Here we have used that $\cup_{j=1}^{n-1} m(\gamma|_j)\subset J_n$ and $K$ is unconditional, so that for any $J\subset I$ and $x\in X$, $|P_J x|_K\leqslant |x|_K$.  

Let $y_\gamma= P_{m(\gamma)\setminus \cup_{j=1}^{|\gamma|-1} m(\gamma|_j)} x_{\theta(\gamma)}$ and $z_\gamma= \pi(y_\gamma)$.  Note that $\|y_\gamma\|$, $\|z_\gamma\|\leqslant 1$.  We claim that for all $0\leqslant \eta\leqslant \zeta$ and all $\gamma\in (\Gamma_\zeta \cup\{\varnothing\})^\eta$, $(z_{\gamma|_j})_{j=1}^{|\gamma|}\in (\hhh_{\ee_0}^L)_{(L, \delta)}^\eta$.  This will imply that $\zeta<o_{(L, \delta)}(\hhh^L_{\ee_0})=\zeta$, and this contradiction will finish the proof.  We prove the result by induction on $\eta$.

Fix $\gamma\in \Gamma_\zeta\cup \{\varnothing\}$.  If $\gamma=\varnothing$, of course $\varnothing \in \hhh^L_{\ee_0}$.  If $\gamma\neq \varnothing$ and $f\in K$ is such that $f(x_{\theta(\gamma)})\geqslant \ee$, \begin{align*} f(y_\gamma) & \geqslant f(x_{\theta(\gamma)}) - |f(x_{\theta(\gamma)}- y_\gamma)| \\ & \geqslant \ee - |f(P_{\cup_{j=1}^{|\gamma|-1} m(\gamma|_j)} x_{\theta(\gamma)})| - |f(P_{I\setminus m(\gamma)}x_{\theta(\gamma)})| \\ & \geqslant \ee - |P_{\cup_{j=1}^{|\gamma|-1} m(\gamma|_j)} x_{\theta(\gamma)}|_K - |P_{I\setminus m(\gamma)}x_{\theta(\gamma)}|_K> \ee- 2\mu = \ee_0.\end{align*} Therefore if $f\in K$ is such that $f(x_{\theta(\gamma|_j)})\geqslant \ee$ for each $1\leqslant j\leqslant |\gamma|$, $f(y_{\gamma|_j})\geqslant \ee_0$ for each $1\leqslant j\leqslant |\gamma|$.  Next, note that if $x=(x_i)_{i\in I}\in X$ and $x^*= (x_i^*)_{i\in I}\in K$, $$x^*(x)=\sum_{i\in I}x^*_i(x_i) \leqslant \sum_{i\in I}\|x^*_i\|\|x_i\|e_i^*(e_i)= \pi_*(x^*)(\pi(x)).$$  Therefore if $f\in K$ is such that $f(y_{\gamma|_j})\geqslant \ee_0$ for each $1\leqslant j\leqslant |\gamma|$, $\pi_*(f)\in L$ and $$\pi_*(f)(z_{\gamma|_j})= \pi_*(f) (\pi (y_{\gamma|_j})) \geqslant \ee_0.$$  This proves that $(z_{\gamma|_j})_{j=1}^{|\gamma|}\in \hhh^L_{\ee_0}$ for each $\gamma\in \Gamma_\zeta$, and gives the base case of the induction.  The limit ordinal case of the induction is trivial.  Assume the result holds for a given $\eta<\zeta$ and suppose $\gamma\in (\Gamma_\zeta \cup\{\varnothing\})^{\eta+1}$.  If $\gamma\neq \varnothing$, write $\gamma=(t, \sigma)$, let $t_0$ be an immediate successor of $t$ in $\mt_\zeta^\eta$, and let $s=(z_{\gamma|_j})_{j=1}^{|\gamma|}$.  If $\gamma=\varnothing$, let $t_0$ be an immediate successor of $t$ in $\mt_\zeta^\eta$ and let $s=\varnothing$.   Then by the inductive hypothesis, for each $J\in [I]^{<\nn}$, $s\cat z_{(t_0, \sigma\cat J)}\in (\hhh^L_{\ee_0})_{(L, \delta)}^\eta$.  But by construction, $P_J z_{(t_0, \sigma \cat J)}=0$, so the net $(z_{(t_0, \sigma\cat J)})_{J\in [I]^{<\nn}}\subset B_E$ is coordinate-wise null.  Since $L\subset [e_i^*:i\in I]$, coordinate-wise nullity of the net $(z_{(t_0, \sigma\cat J)})_{J\in [I]^{<\nn}}$ implies it is pointwise null on $L$.  Therefore for any $U\in (L, \delta)$, the net $(z_{(t_0, \sigma\cat J)})_{J\in [I]^{<\nn}}$ is eventually in $U$, whence there exists $z\in U$ so that $s\cat z\in (\hhh^L_{\ee_0})_{(L, \delta)}^\eta$.  Since this holds for any $U\in (L, \delta)$, $s\in (\hhh^L_{\ee_0})_{(L, \delta)}^{\eta+1}$, finishing the induction.

\end{proof}

\begin{corollary} Let $(e_i)_{i\in I}$, $(f_i)_{i\in I}$ be $1$-unconditional bases for the Banach spaces $E, F$.  Suppose also that for each $i\in I$, $B_i:X_i\to Y_i$ is an operator so that the function $e_i\mapsto \|B_i\|f_i$ extends to an operator $B:E\to F$.  Then $A:(\oplus_{i\in I}X_i)_E\to (\oplus_{i\in I}Y_i)_F$ defined by $A(x_i)_{i\in I}=(B_ix_i)_{i\in I}$ is an operator satisfying $Sz(A)\leqslant (\sup_{i\in I} Sz(B_i))Sz(A)$.  In particular, if $E=F$ and $(e_i)_{i\in I}=(f_i)_{i\in I}$, and if $\sup_{i\in I} \|B_i\|<\infty$, $Sz(A)\leqslant (\sup_{i\in I}Sz(B_i)) Sz(E)$.

\end{corollary}

\begin{proof}   Let $L_i=B^*_iB_{Y_i^*}$ and $L=B^*B_{F^*}$.  Assume first that $L\subset [e_i^*:i\in I]$.  We can apply Theorem \ref{infinite direct sum}, since $A^*B_{(\oplus_{i\in I} Y_i)^*_F}\subset\{x^* \in X^*\cap \prod_{i\in I}L_i: \pi_*(x^*)\in L\}$, and we finish immediately. To see this inclusion, we first fix $(y_i^*)_{i\in I}\subset B_{(\oplus_{i\in I}Y_i)^*_F}$ and note that the formal series $\sum_{i\in I} \|y^*_i\|f_i^*\in B_{F^*}$.  It is easy to see that $A^*(y_i^*)_{i\in I}= (B_i^*y_i^*)_{i\in I}$, and $$\pi_* A^*(y_i)_{i\in I} = \sum_{i\in I} \|B_i^*y_i^*\|e_i^* \leqslant_{\text{pt}} \sum_{i\in I}\|B_i^*\|\|y_i^*\|e_i^* = B^*\sum_{i\in I}\|y_i^*\|f_i^*\in B^*B_{F^*}.$$ Here, $\leqslant _{\text{pt}}$ denotes coordinate-wise domination.  Since $B_{F^*}$ and therefore $B^*B_{F^*}$, are closed under pointwise suppression, the pointwise suppression $\pi_*A^*(y_i^*)_{i\in I}$ of $B^*\sum_{i\in I}\|y^*_i\|f_i^*\in B^*B_{F^*}$ also lies in $B^*B_{F^*}$.

If $L\not\subset [e_i^*:i\in I]$, then the operator $B:E\to F$ preserves a copy of $\ell_1$, and $I(B)=Sz(B)=\infty$.

\end{proof}

\subsection{Subspace and quotient estimates}

The main result of this subsection is the following.

\begin{theorem} There exists a constant $C>1$ so that if $X$ is any Banach space and $Y$ is any subspace of $X$, $$Sz_\ee(B_{X^*})\leqslant Sz_{\ee/C}(B_{(X/Y)^*}) Sz_{\ee/C}(B_{Y^*}).$$   In particular, $Sz(X)\leqslant Sz(X/Y)Sz(Y)$.  Moreover, for any ordinal $\xi$, $Sz(\cdot)<\omega^{\omega^\xi}$ and $Sz(\cdot)\leqslant \omega^{\omega^\xi}$ are three space properties.  

\label{three space}
\end{theorem}

\begin{remark}

In \cite{L0}, it was shown that in the case that $Sz(Y), Sz(X/Y)<\omega_1$, $Sz(X)\leqslant \omega Sz(X/Y) Sz(Y)$ using the slicing definition of the Szlenk index. In \cite{BL}, it was shown that in the case that $Sz(Y), Sz(X/Y)<\omega_1$, $Sz(X)\leqslant Sz(X/Y) Sz(Y)$, also using the slicing definition.  Our proof establishes this result without any assumptions on $Sz(Y)$, $Sz(X/Y)$.

\end{remark}

\begin{lemma} For any subspace $Y$ of $X$, $Sz(Y), Sz(X/Y)\leqslant Sz(X)$.  

\label{three space lemma}

\end{lemma}

For the proofs in this subsection, recall that for a Banach space $Z$, $\mmm(Z)$ denotes our specified weak neighborhood basis of zero in the Banach space $Z$.  

\begin{proof} Of course, it is trivial to see that for any $\ee>0$ and any $\xi\in \ord$, $(\hhh^{B_{Y^*}}_\ee)_w^\xi\subset (\hhh^{B_{X^*}}_\ee)^\xi_w$, so $$Sz(Y)=\sup_{\ee>0}o_w(\hhh^{B_{Y^*}}_\ee) \leqslant \sup_{\ee>0} o_w(\hhh^{B_{X^*}}_\ee) = Sz(X).$$  

Next, note that if $(w_V)_{V\in \mmm(X/Y)}\subset B_{X/Y}$ is a weakly null net and if $U\in \mmm(X)$, there exist $V_1\in \mmm(X/Y)$ and $x\in 5B_X$ so that $x\in U$ and $Qx=w_{V_1}$. Here, $Q:X\to X/Y$ is the quotient map.  To see this, first choose $(x_V)_{V\in \mmm(X/Y)}\subset 2 B_X$ so that $Qx_V=w_V$ for all $V\in \nnn$.  By passing to a subnet $(x_V)_{V\in D}$, we may assume that $x_{V_1}-x_{V_2}\in \frac{1}{2}U$ for all $V_1, V_2\in D$. Choose $\ee\in (0,1)$ so that $\ee B_X\subset \frac{1}{2}U$.   Since $(w_V)_{V\in D}$ is a weakly null net, there exists a convex combination $w$ of $(w_V)_{V\in D}$ with $\|w\|<\ee$.  Let $u_1$ be the corresponding convex combination of $(x_V)_{V\in D}$. Note that $\|u_1\|\leqslant 2$ and $Qu_1=w$.  Fix $V_1\in D$ and let $u_2=x_{V_1}- u_1$.  Since $U$ is convex, and since $u_2$ is a convex combination of members of $\frac{1}{2}U$, $u_2\in \frac{1}{2}U$.  Moreover, $\|x_{V_1}\|, \|u_1\|\leqslant 2$, so $\|u_2\|\leqslant 4$. Also, $\|Qu_2- w_{V_1}\|= \|w\|<\ee$.  Choose $u_3\in X$ with $\|u_3\|<\ee$ so that $Qu_3=w$.  Then $u_2+ u_3\in \frac{1}{2}U+ \frac{1}{2}U=U$ and $Q(u_2+u_3)= w_{V_1}.$  Taking $x=u_2+u_3$ finishes the claim.  

Next, we claim that if $s\in \hhh^{B_{(X/Y)^*}}_\ee$ is such that $o_w(\hhh^{B_{(X/Y)^*}}_\ee(s))>\xi$, there exists a collection $(x_\alpha)_{\alpha\in \aaa_\xi}\subset 5B_X$ so that for each $\alpha\in \aaa_\xi$,  \begin{enumerate}[(i)]\item $(Qx_{\alpha|_i})_{i=1}^{|\alpha|}\in \hhh^{B_{(X/Y)^*}}_\ee(s)$, \item if $\alpha=(t, (U_1, \ldots, U_n))$, $x_\alpha \in U$.  \end{enumerate}  In particular, taking $s=\varnothing$, we deduce that if $o_w(\hhh^{B_{(X/Y)^*}}_\ee)>\xi$, there exists $(x_\alpha)_{\alpha\in \aaa_\xi}$ satisfying properties (i) and (ii). By Lemma \ref{lemma2}, we deduce that $(x_\alpha/5)_{\alpha\in \aaa_\xi}$ witnesses the fact that $o_w(\hhh^{B_{X^*}}_{\ee/5})>\xi$, which finishes the proof once we have the claim.  

Of course, the proof of the claim is by induction.  The $\xi=0$ and $\xi$ a limit cases are trivial.  Assume the result holds for some $\xi$ and suppose $o_w(\hhh^{B_{(X/Y)^*}}_\ee(s))>\xi+1.$  This means we can find a weakly null net $(w_V)_{V\in \mmm(X/Y)}\subset B_{X/Y}$ so that $o_w(\hhh^{B_{(X/Y)^*}}_\ee(s\cat w_V))>\xi$ for all $V\in \mmm(X/Y)$.  For a given $U\in \mmm(X)$, using the claim above, we can choose $x_U\in U$ and $V_U\in \mmm(X/Y)$ with $\|x_U\|\leqslant 5$ so that $Qx_U=w_{V_U}$.  Applying the inductive hypothesis to $s\cat w_{V_U}$, we deduce the existence of some $(x^U_\alpha)_{\alpha\in \aaa_\xi}$ satisfying (i) and (ii) with $s$ replaced by $s\cat w_{V_U}$.  We then define $(x_\alpha)_{\alpha\in \aaa_\xi}$ by letting $$x_{(\xi+1, U)}=x_U,$$ $$x_{(\xi+1, U)\cat (t, U)}= x^U_{(t, \sigma)}$$ for $t\in \ttt_\xi$.

\end{proof}

\begin{proof}[Proof of Theorem \ref{three space}] Recall that for $\delta>0$ we let $$(B_{Y^*}, \delta) =\Bigl\{\{y\in Y: |y^*(y)|< \delta\text{\ }\forall y^*\in F\}: F\subset B_{Y^*} \text{\ finite}\Bigr\}. $$  We will show that for any $\ee\in (0,1)$ and any $0<\rho < \delta<\ee-\rho$, \begin{equation} o_w(\hhh^{B_{X^*}}_\ee)\leqslant o_w(\hhh^{B_{(X/Y)^*}}_\rho)o_{(B_{Y^*},\delta/2)}(\hhh^{B_{Y^*}}_{(\ee-\rho)/2}).\tag{1}\end{equation} We first assume the inequality $(1)$ and prove the theorem, and then return to the proof of $(1)$.

The proof of Theorem \ref{main theorem} yields that $$Sz_{5\ee}(B_{X^*})\leqslant o_w(\hhh^{B_{X^*}}_\ee) \leqslant o_w(\hhh^{B_{(X/Y)^*}}_\rho) o_{(B_{Y^*}, \delta/2)}(\hhh^{B_{Y^*}}_{(\ee-\rho)/2}) \leqslant Sz_{\rho/2}(B_{(X/Y)^*}) Sz_{(\ee-\rho-\delta)/4}(B_{Y^*}).$$ Setting $\rho=\ee/4$ and $\delta=2\rho$ yields the first statement of the theorem with $C=80$.  

It follows from Lemma \ref{three space lemma} that if $Sz(X)<\omega^{\omega^\xi}$ (resp. $Sz(X)\leqslant \omega^{\omega^\xi}$), the same inequality holds for both $Sz(Y)$ and $Sz(X/Y)$.  If $Sz(Y)$, $Sz(X/Y)< \omega^{\omega^\xi}$, $(1)$ immediately yields that $$o_w(\hhh^{B_{X^*}}_\ee)\leqslant Sz(X/Y)Sz(Y)<\omega^{\omega^\xi}$$ for any $\ee\in (0,1)$, which yields that $Sz(\cdot)<\omega^{\omega^\xi}$ is a three space property.  Here we have used the fact that if $\zeta, \eta<\omega^{\omega^\xi}$, $\zeta\eta<\omega^{\omega^\xi}$.  If $Sz(Y), Sz(X/Y)\leqslant \omega^{\omega^\xi}$, then for any $\ee\in (0,1)$, choose any $0<\rho<\delta < \ee-\rho$ and note that by Proposition \ref{special form}$(iii)$, $o_w(\hhh_\rho^{B_{(X/Y)^*}}), o_{(B_{Y^*}, \delta/2)}(\hhh^{B_{Y^*}}_{(\ee-\rho)/2})$ must be strictly less than $\omega^{\omega^\xi}$.  Then inequality $(1)$ gives that $o_w(\hhh^{B_{X^*}}_\ee)< \omega^{\omega^\xi}$, and $Sz(\cdot)\leqslant \omega^{\omega^\xi}$ is a three space property.

We now return to the proof of $(1)$.  Let $\xi=o_w(\hhh^{B_{(X/Y)^*}}_\rho)$ and assume $\xi\in \ord$ (otherwise the result is trivial).  We claim that for any $\zeta\in \ord$ and any $s\in \hhh^{B_{X^*}}_\ee$ so that $o_w(\hhh^{B_{X^*}}_\ee(s))>\xi\zeta$, there exists $(x_\alpha)_{\alpha\in \aaa_\zeta}$ so that for all $\alpha\in \aaa_\zeta$, \begin{enumerate}[(i)]\item $\|x_\alpha\|_{X/Y}< \rho$, \item $(x_{\alpha|_i})_{i=1}^{|\alpha|}\in \hhh^K_\ee(s)$, \item if $\alpha=(t, (U_1, \ldots, U_n))$, $x_\alpha\in U_n$. \end{enumerate}  The $\zeta=0$ and $\zeta$ a limit ordinal case are trivial.  Assume the result holds for a given $\zeta$ and assume $s\in \hhh^{B_{X^*}}_\ee$ is such that $o_w(\hhh^{B_{X^*}}_\ee(s))>\xi\zeta+\xi=\xi(\zeta+1)$.  We will show that for each $U\in \mmm(X)$, there exists $x_U\in U\cap B_X$ so that $o_w(\hhh^{B_{X^*}}_\ee(s\cat x_U))>\xi\zeta$ and $\|x_U\|_{X/Y}<\rho$.  Let $\hhh=(\hhh^{B_{X^*}}_\ee)_w^{\xi\zeta}(s)$ and note that $o_w(\hhh)>\xi$.  By Lemma \ref{lemma2}, we can fix $(z_\alpha)_{\alpha\in \aaa_\xi}$ so that for each $\alpha\in \aaa_\xi$, \begin{enumerate}[(i)]\item $(z_{\alpha|_i})_{i=1}^{|\alpha|}\in \hhh$, \item if $\alpha=(t, (U_1, \ldots, U_n))$, $z_\alpha\in U_n$. \end{enumerate} By replacing $z_{(t, (U_1, \ldots, U_n))}$ with $z_{(t, (U\cap U_1, \ldots, U\cap U_n))}$, we may assume $z_\alpha\in U\cap B_X$ for all $\alpha\in \aaa_\xi$.  If for each $\alpha\in \aaa_\xi$, every convex combination $z$ of $(z_{\alpha|_i})_{i=1}^{|\alpha|}$ is such that $\|z\|_{X/Y}\geqslant \rho$, we claim that $(Qz_\alpha)_{\alpha\in \aaa_\xi}$ would imply that $o_w(\hhh^{B_{(X/Y)^*}}_\rho)>\xi$, which would be a contradiction.  To see this, we claim that for every $0\leqslant \eta\leqslant \xi$ and every $\alpha\in (\aaa_\xi\cup \varnothing)^\eta$, $(Qz_{\alpha|_i})_{i=1}^{|\alpha|}\in (\hhh^{B_{(X/Y)^*}}_\rho)^\eta_w$.  The $\eta=0$ case and the $\eta$ a limit ordinal cases are clear.  Suppose $\alpha=(t, \sigma)\in \aaa_\xi^{\eta+1}$, which happens if and only if $t\in \ttt_\xi^{\eta+1}$.  Let $t_0$ be an immediate successor of $t$ in $\ttt_\xi^\eta$.  Then for every $U\in \mmm(X)$, $(Qz_{\alpha|_i})_{i=1}^{|\alpha|}\smallfrown  Qz_{(t_0, \sigma\cat U)}\in (\hhh^{B_{(X/Y)^*}}_\rho)^\eta_w$ and, since $(Qz_{(t_0, \sigma\cat U)})_{U\in \mmm(X)}$ is a weakly null net, we deduce $(Qz_{\alpha|_i})_{i=1}^{|\alpha|}\in ((\hhh^{B_{(X/Y)^*}}_\rho)_w^\eta)_w'= (\hhh^{B_{(X/Y)^*}}_\rho)_w^{\eta+1}$.  This completes the inductive proof, and guarantees that there must exist some convex combination $x_U$ of some $(z_{\alpha|_i})_{i=1}^{|\alpha|}$ so that $\|x_U\|_{X/Y}<\rho$.  But since $s\cat (z_{\alpha|_i})_{i=1}^{|\alpha|}\in (\hhh^{B_{X^*}}_\ee)^{\xi\zeta}_w$, the convex block $s\cat x_U$ lies in $(\hhh^{B_{X^*}}_\ee)^{\xi\zeta}_w$ as well.  This implies that $o_w(\hhh^{B_{X^*}}_\ee(s\cat x_U))>\xi\zeta$.  By the inductive hypothesis, this means that for each $U\in \mmm(X)$, there exists $(x^U_\alpha)_{\alpha\in \aaa_\zeta}$ satisfying (i)-(iii) with $s$ replaced by $s\cat x_U$.  We define $(x_\alpha)_{\alpha \in \aaa_{\zeta+1}}$ by letting $$x_{((\zeta+1), U)}= x_U$$ and $$x_{(\zeta+1, U)\cat (t,\sigma)}= x^U_{(t, \sigma)}.$$  This completes the claim.  

We last show that for $0<\zeta$, if $(x_\alpha)_{\alpha\in \aaa_\zeta}$ satisfies (i)-(iii) of the previous paragraph, and if $(y_\alpha)_{\alpha\in \aaa_\zeta}\subset Y$ is chosen so that $\|x_\alpha-y_\alpha\|<\rho$ for all $\alpha\in \aaa_\zeta$, then for any $0\leqslant \eta\leqslant \zeta$, for any $\alpha\in (\aaa_\zeta\cup \{\varnothing\})^\eta$, $(y_{\alpha|_i}/2)_{i=1}^{|\alpha|}\in (\hhh^{B_{Y^*}}_{(\ee-\rho)/2})^\eta_{(B_{Y^*}, \delta/2)}$.  This will imply that $\varnothing \in (\hhh^{B_{Y^*}}_{(\ee-\rho)/2})^\zeta_{(B_{Y^*}, \delta/2)}$ and $o_{(B_{Y^*}, \delta/2)}(\hhh^{B_{Y^*}}_{(\ee-\rho)/2})>\zeta$, yielding $(1)$. First note that $\|y_\alpha\|\leqslant \rho + \|x_\alpha\|\leqslant 2$.  For the base case, since $(y_{\alpha|_i}/2)_{i=1}^{|\alpha|}$ is a $\rho/2$-perturbation of $(x_{\alpha|_i}/2)_{i=1}^{|\alpha|}$, $(y_{\alpha|_i}/2)_{i=1}^{|\alpha|}\in \hhh^{B_{Y^*}}_{(\ee-\rho)/2}$.  The limit ordinal case is trivial.  Assume the result holds for a given $\eta<\zeta$.   Fix $0<\mu<\delta-\rho$ and $$U=\{y\in Y: |y^*(y)|<\delta/2 \text{\ }\forall y^*\in F\}\in (B_{Y^*}, \delta/2)$$ for some finite set $F\subset B_{Y^*}$. Let $E\subset B_{X^*}$ consist of Hahn-Banach extensions of each member of $F$ and let $$V=\{x\in X: |x^*(x)|<\mu \text{\ }\forall x^*\in F\}\in \mmm(X).$$  Fix $\alpha\in (\aaa_\zeta\cup \{\varnothing\})^{\eta+1}$. If $\alpha=\varnothing$, let $t_0$ be an immediate successor of $\varnothing$ in $\mt_\zeta^\eta$.  If $\alpha\neq \varnothing$, write $\alpha=(t, \sigma)$ and let $t_0$ be an immediate successor of $t$ in $\mt_\zeta^\eta$.  Then since $x_{(t_0, \sigma\cat V)}\in V$ and $\|x_{(t_0, \sigma\cat V)} - y_{(t_0, \sigma\cat V)}\|<\rho$, we deduce that $y_{(t_0, \sigma \cat V)}/2\in U$ and $(y_{\alpha|_i})_{i=1}^{|\alpha|}\smallfrown y_{(t_0, \sigma \cat V)}\in (\hhh^{B_{Y^*}}_{(\ee-\rho)/2})^\eta_{(B_{Y^*}, \delta/2)}$. Since $U$ was an arbitrary member of $(B_{Y^*}, \delta/2)$, we deduce $(y_{\alpha|_i})_{i=1}^{|\alpha|}\in (\hhh^{B_{Y^*}}_{(\ee-\rho)/2})_{(B_{Y^*}, \delta/2)}^{\eta+1}$, which finishes the proof.

\end{proof}

\subsection{Constant reduction}

The main result of this subsection is the following. 

\begin{theorem} Let $X$ be a Banach space and $K=B_{X^*}$.  Then for any $\delta, \ee\in (0,1)$, $$o_w(\hhh^K_{\delta\ee})\leqslant o_w(\hhh^K_\delta)o_w(\hhh^K_\ee).$$ In particular, if $\xi\in \emph{\ord}$ and $o_w(\hhh^K_\ee)>\omega^{\omega^\xi}$ for some $\ee\in (0,1)$, then $o_w(\hhh^K_\ee)>\omega^{\omega^\xi}$ for every $\ee\in (0,1)$.  If $\xi$ is a limit ordinal, then $Sz(X)\neq \omega^{\omega^\xi}$.

\label{constant reduction}
\end{theorem}

\begin{remark} It was shown in \cite{L1} that with $K=B_{X^*}$, $Sz_{\delta\ee}(K)\leqslant Sz_\delta(K)Sz_\ee(K)$ for any $\ee, \delta\in (0,1)$.  This inequality and $o_w(\hhh^K_{\delta\ee})\leqslant o_w(\hhh^K_\ee)o_w(\hhh^K_\delta)$ can both be used to prove the remaining statements of Theorem \ref{constant reduction}, but these inequalities do not imply each other from Theorem \ref{main theorem}.  Indeed, examining the proof of Theorem \ref{main theorem}, the first inequality of Theorem \ref{constant reduction} can only be used to prove that there exists a constant $C>1$ so that $Sz_{\delta\ee}(K)\leqslant Sz_{\ee/C}(K)Sz_{\delta/C}(K)$.  Similarly, the inequality $Sz_{\delta\ee}(K)\leqslant Sz_\ee(K)Sz_\delta(K)$ combined with Theorem \ref{main theorem} only yields a weakened version of the first inequality of Theorem \ref{constant reduction} involving a constant.

\end{remark}

\begin{proof}[Proof of Theorem \ref{constant reduction}] We first assume the first inequality of the theorem and complete the proofs of the remaining statements.  Fix $\delta, \ee\in (0,1)$ and assume $o_w(\hhh^K_\ee)>\omega^{\omega^\xi}$.  Assume $o_w(\hhh^K_\delta)=\zeta<\omega^{\omega^\xi}$ and fix $n\in \nn$ so that $\delta^n<\ee$.  Then $$\omega^{\omega^\xi}< o_w(\hhh^K_\ee) \leqslant o_w(\hhh^K_{\delta^n})\leqslant o_w(\hhh^K_\delta)^n=\zeta^n<\omega^{\omega^\xi},$$ a contradiction.  Thus $o_w(\hhh^K_\delta)\geqslant \omega^{\omega^\xi}$, but since the index $o_w(\hhh^K_\delta)$ must be a successor, this inequality is strict.  

Next, suppose $\xi$ is a limit ordinal and $Sz(K)\geqslant \omega^{\omega^\xi}$.  Then for any $\zeta<\xi$, there exists $\ee>0$ so that $o_w(\hhh^K_\ee)>\omega^{\omega^\zeta}$ and, by the previous paragraph, $o_w(\hhh^K_{1/2})>\omega^{\omega^\zeta}$.  Since this holds for every $\zeta<\xi$, $o_w(\hhh^K_{1/2})\geqslant \omega^{\omega^\xi}$, and again this inequality must be strict.  Therefore $Sz(K)>\omega^{\omega^\xi}$, and there is no Banach space with Szlenk index $\omega^{\omega^\xi}$.  

We last turn to the proof of the first inequality of the theorem. Of course, is suffices to consider the case that $o_w(\hhh^K_\delta), o_w(\hhh^K_\ee)\in \ord$. Let $\xi=o_w(\hhh^K_\delta)$.    We will show the following claim:  If $s\in \hhh^K_{\delta\ee}$ is such that $o_w(\hhh^K_{\delta\ee}(s))>\xi \zeta$, then there exists $(x_\alpha)_{\alpha\in \aaa_\zeta}\subset B_X$ so that for each $\alpha\in \aaa_\zeta$, \begin{enumerate}[(i)]\item $\|x_\alpha\|<\delta$ , \item $(x_{\alpha|_i})_{i=1}^{|t|}\in \hhh^K_{\delta\ee}(s)$, \item if $\alpha=(t, (U_1, \ldots, U_n))$, $x_\alpha\in U_n$. \end{enumerate} 

Applying this with $s=\varnothing$ gives that if $o_w(\hhh^K_{\delta\ee})>\xi\zeta$, there exists $(x_\alpha)_{\alpha\in \aaa_\zeta}$ satisfying properties (i)-(iii).  Then it is easy to verify that $(\delta^{-1} x_\alpha)_{\alpha\in \aaa_\zeta}\subset B_X$ witnesses the fact that $o_w(\hhh^K_\ee)>\zeta$, using homogeneity and Lemma \ref{lemma2}.  Therefore if the inequality were to fail, we could set $\zeta=o_w(\hhh^K_\ee)$ and obtain a contradiction.

The claim is trivial for $\zeta=0$ or $\zeta$ a limit ordinal.  Assume the claim holds for a given $\zeta$ and assume $s\in \hhh^K_{\delta\ee}(s)$ is such that $o_w(\hhh^K_{\delta\ee}(s))>\xi(\zeta+1)=\xi\zeta+\xi$ for some $s\in \hhh^K_{\delta\ee}$.  We will show that for any $U\in \mmm$, there exists $x_U\in \delta B_X\cap U$ so that $o_w(\hhh^K_{\delta\ee}(s\cat x_U))>\xi\zeta$.  Let $\hhh=(\hhh^K_{\delta\ee}(s))^{\xi\zeta}_w$ and note that since $o_w(\hhh^K_{\delta\ee}(s))>\xi\zeta+\xi$, $o_w(\hhh)>\xi$.  This means we can find $(y_\alpha)_{\alpha\in \aaa_\xi}$ so that for each $\alpha\in \aaa_\xi$, $(y_{\alpha|_i})_{i=1}^{|\alpha|}\in \hhh$ and for $\alpha=(t, (U_1, \ldots, U_n))$, $y_\alpha\in U_n$.  By replacing $y_{(t, (U_1, \ldots, U_n))}$ with $y_{(t, (U\cap U_1, \ldots, U\cap U_n))}$, we may assume that $(y_\alpha)_{\alpha\in \aaa_\xi}\subset U\cap B_X$.   If every convex combination of $(y_{\alpha|_i})_{i=1}^{|\alpha|}$ has norm at least $\delta$, Lemma \ref{lemma2} can be applied to $(y_\alpha)_{\alpha\in \aaa_\xi}$ to deduce that $o_w(\hhh^K_\delta)>\xi$, a contradiction.  Therefore there exist $\alpha\in \aaa_\xi$ and a vector $x_U$ which is a convex combination of $(y_{\alpha|_i})_{i=1}^{|\alpha|}$ which has norm less than $\delta$. By convexity of $U$, $x_U\in U$. Since $(y_{\alpha|_i})_{i=1}^{|\alpha|}\in \hhh$, $s\cat (y_{\alpha|_i})_{i=1}^{|\alpha|}\in  (\hhh^K_{\delta\ee})^{\xi\zeta}_w$.  Since $s\cat x_U$ is a convex block of $s\cat (y_{\alpha|_i})_{i=1}^{|\alpha|}\in (\hhh^K_{\delta\ee})^{\xi\zeta}_w$, $s\cat x_U\in (\hhh^K_{\delta\ee})^{\xi\zeta}_w$, and $o_w(\hhh^K_{\delta\ee}(s\cat x_U))>\xi\zeta$.  

Next, for each $U\in \mmm$ and $x_U\in \delta B_X\cap U$ with $o_w(\hhh^K_{\delta\ee})>\xi\zeta$, we use the inductive hypothesis to find $(x^U_\alpha)_{\alpha\in \aaa_\zeta}$ satisfying (i)-(iii) with $s$ replaced by $s\cat x_U$.  Then define $(x_\alpha)_{\alpha\in \aaa_{\zeta+1}}$ by letting $x_{(\zeta+1, U)}=x_U$ and $x_{(\zeta+1, U)\cat (t,\sigma)}=x^U_{(t, \sigma)}$ for $t\in \ttt_\zeta$.  It is easy to verify that $(x_\alpha)_{\alpha\in \aaa_{\zeta+1}}$ satisfies (i)-(iii).

\end{proof}

\end{document}